\documentclass{amsart}
\usepackage{graphicx}
\numberwithin{equation}{section}
\newtheorem{theorem}{Theorem}[section]
\newtheorem{lemma}[theorem]{Lemma}
\newtheorem{proposition}[theorem]{Proposition}

\theoremstyle{definition}

\newtheorem{remark}[theorem]{Remark}

\newcommand{\eps}{\varepsilon}
\newcommand{\la}{\lambda}

\renewcommand{\qed}{\hfill $\square$}

\def\bx{\mathbf x}

\def\CC{\mathbb C}
\def\ZZ{\mathbb Z}
\def\RR{\mathbb R}

\def\HH{\mathbb H}

\def\B{\mathcal B}

\def\R{\mathcal R}

\def\sX{\mathbf X}

    \def\La{\Lambda} %% pressure

\date{September 26, 2008} %Printed: \today\ from file: \jobname.tex}
\begin{document}
\title[LDP for leaves in random trees]{Large deviations for the leaves in some random trees}

\author{Wlodek Bryc}
\address{
Department of Mathematical Sciences, University of Cincinnati, 2855
Campus Way, PO Box 210025, Cincinnati, OH 45221-0025, USA}
\thanks{Research supported in part by Taft Research Seminar 2006/07 and 2008/09.}
\email{Wlodzimierz.Bryc@UC.edu}

\author{David Minda}
\thanks{Research supported in part by Taft Research Seminar 2008/09.}
\address{
Department of Mathematical Sciences, University of Cincinnati, 2855
Campus Way, PO Box 210025, Cincinnati, OH 45221-0025, USA}
\email{David.Minda@math.uc.edu }

\author{Sunder Sethuraman}
\thanks{ Research  supported in part by NSF
grant DMS-0504193.}

\address{
Department of Mathematics, 396 Carver Hall, Iowa State University, Ames, IA 50011, USA
}
\email{sethuram@iastate.edu}
%
%\authorone[University of Cincinnati]{Wlodek Bryc } % Affiliation is just the name of your university or institution

%\addressone{Department of Mathematical Sciences, University of Cincinnati, 2855
%Campus Way, PO Box 210025, Cincinnati, OH 45221-0025, USA. \email{Wlodzimierz.Bryc@UC.edu}} % Your postal address goes here.
%\authortwo[University of Cincinnati]{David Minda } % Affiliation is just the name of your university or institution

%\addresstwo{Department of Mathematical Sciences, University of Cincinnati, 2855
%Campus Way, PO Box 210025, Cincinnati, OH 45221-0025, USA. \email{David.Minda@math.uc.edu }} % Your postal address goes here.

%\authorthree[Iowa State University]{Sunder Sethuraman} % Affiliation is just the name of your university or institution

%\addressthree{Department of Mathematics, 396 Carver Hall, Iowa State University, Ames, IA 50011, USA. \email{sethuram@iastate.edu}}  
%SS changed room to 396
% Your postal address goes here.

%%% ams title
%\title[LDP for leaves in some random graphs]
%{Large deviations for the leaves in some preferential attachment and
%recursive graphs}

\begin{abstract}
Large deviation principles and related results are given for a
class of Markov chains associated to the ``leaves" in random recursive
trees and preferential attachment random graphs, as well as the ``cherries'' in Yule trees.
 In particular, the method of proof, combining
analytic and Dupuis-Ellis type path arguments, allows for an explicit computation of the
large deviation pressure.
\end{abstract}

%% APT:
%\ams{60F10}{05C80}
%S large deviations primary; random graphs secondary NICE


\maketitle
\section{Introduction and results}

We consider in this article large deviations and related laws of
large numbers and central limit theorems for a class of Markov
chains associated to the number of leaves, or nodes of degree one, in preferential attachment random graphs and
random recursive trees, and also the number of ``cherries," or pairs of leaves with a common parent, in Yule trees.
The random
graphs studied model various networks such as pyramid schemes, chemical
polymerization, the internet, social structures, genealogical families, among others.
In particular, the leaf and cherry counts in these models are of interest, and
have concrete interpretations. In Subsection \ref{Models}, we
discuss applications with these models and literature.

Define the nondecreasing Markov chain $\{Z_n: n\geq 1\}$ starting
from initial state $Z_1=k_0\geq 0$,
by its one-step
transitions, for $n\geq 1$,
\begin{equation}
  \label{Z-def}
  \Pr\big(Z_{n+1}-Z_n=v|Z_n\big)\ =\ \begin{cases}
1-\frac{Z_n}{s_n} & \mbox{ if $v=1$} \\
\frac{Z_n}{s_n} & \mbox{ if $v=0$} \\
  \end{cases}
\end{equation}
where $\{s_n: n\geq 1\}$ is a sequence of positive numbers such that
\begin{equation}
  \label{s_n}
  s_n\geq k_0+n-1, \mbox{ and } \frac{s_n}{n} \rightarrow \alpha  \mbox{ for } 1<
\alpha<\infty
\end{equation}
with convention $0/0 = 0$.
Additionally, we also consider two special sequences, $s_n=n$
and $s_n=n/2$ with $k_0\leq s_1$,
related to some applications.

We first note, in this form, $Z_n$ can be seen to represent the ``mass'' of red
balls in a Polya-Eggenberger-type urn of two colors, red and blue,
not necessarily tenable, where at each time $n$, proportional to the
red balls mass, a (signed) mass $s_{n+1}-s_n$ of blue balls is
deposited, otherwise one red ball and mass $s_{n+1}-s_n -1$ of blue
balls is added.

Also, when $s_n = (k_0 + 2(n-1) +n\beta)/(1+\beta)$ and
$\alpha= (2+\beta)/(1+\beta)$ for $\beta>-1$, the chain has
interpretation as the number of leaves in preferential attachment
random graphs with weight function $f(k)=k+\beta$. Later, in
Subsection \ref{Models},
%where we elaborate more on these models,
we also remark that $s_n$
can be taken as a random sequence with respect to the number of ``generalized" leaves, or ``buds,"
that is those nodes, possibly with degree greater than one, which however connect to only one other vertex,
in preferential
attachment graphs with random edge additions.

In addition, when $s_n = \alpha n$ for $\alpha = 1,2$,
$Z_n$ is also the number of
leaves in uniformly and planar oriented recursive trees respectively.
In the latter case $\alpha =2$, $Z_n$ is also
the number of ``plateaux" in a random Stirling
permutation of length $2n$.

Moreover, when $s_n = n/2$, $Z_n$ can be seen as the number of cherries in Yule trees.

As the urns add mass of the opposite color,
%are of `Bernard Friedman' type,
one should expect almost sure limits for the mean behavior and
Normal fluctuations.  In fact, by mostly martingale, combinatorial and urn methods, laws of large numbers (LLN) and
central limit theorems (CLT) have been proved, at least in the
examples mentioned above when $s_n$ is linear with slope $\alpha$ (cf. Theorem \ref{T4}).
See
 \cite{BRST-01,Mori-01,Rudas-Toth-Valko-07}
with respect to preferential attachment,
\cite{Janson04}, \cite{Mahmoud-Smythe} with respect to recursive trees, \cite{McKenzie-Steel-00} with respect to Yule trees, and
\cite{Kotz-Balakrishnan} with respect to urns when $\alpha$ is a
positive integer.

Characterizing the associated large deviations is a natural problem which gives insight into the properties of rare events,
and seems less studied in urns or random graphs.  Previous work has
concentrated on analytic methods with respect to ``subtraction" urn
models--not applicable in our general setting
\cite{Flajolet-Gabarro-Pekari-05}--or extensions of the Dupuis-Ellis
weak convergence approach (cf. \cite{Dupuis-Ellis-97}) to different
allocation models than ours \cite{Dupuis-Nuzman-Whiting-04}, \cite{Zhang-Dupuis-08}.
We note also
some exponential bounds via martingale concentration inequalities
are found in the case $s_n$ is linear with slope $\alpha$
\cite{Chung-Lu}.  See also
 \cite{bakhtin-2008a,Benassi-96,Broutin-Devroye-06,Dembo-Morters-Sheffield-05,dereich-2008,Jabbour-Hattab}
 for other types of large deviations work in various random tree models.

In this context, our main results are to prove a large deviation
principle (LDP) for $Z_n/n$ with an explicitly computed ``pressure,''
or Legendre transform of the associated rate function (Theorem
\ref{T3}). This is done in two different ways for the important case
$s_n$ is linear with $\alpha =2$. Such explicit computations are not commonplace,
and our ``ODE'' method is quite different from the methods in
\cite{Flajolet-Gabarro-Pekari-05} where a quasi-linear PDE is
solved, or in \cite{Zhang-Dupuis-08} where a finite-dimensional
minimization problem is obtained.
%numerically treated.

Perhaps a main feature of our work is that the method given appears
robust and applicable in diverse, not necessarily urn settings.  In
particular, we show the LDP for $Z_n/n$ does not feel the disorder in
the sequence $s_n$, is not dependent on the initial value $Z_1$, and
is the same as for the chain with a regular, linear $s_n$ with slope
$\alpha$. We mention that this is a consequence of a large deviation
principle for the path interpolation of $Z_{\lfloor nt\rfloor}/n$ (Theorem \ref{T2}),
perhaps of interest in itself, that we establish, by the
Dupuis-Ellis weak convergence approach.

In addition, aside from laws of large numbers which are trivially
obtained, we prove a central limit theorem for $Z_n$ through complex
variables arguments with the pressure (Theorem \ref{T4}). These
alternate proofs of the LLN and CLT, although indirect, apply when
the mass additions take on some negative non-integer values where
less is known in the literature. Moreover, the results
%S made consistent with buds improve the
%weak LLN in \cite{athreya-2007} to a strong law, and
give a ``quenched'' LLN, CLT and LDP's with respect to ``generalized" leaves or buds in a
preferential attachment scheme with random edge additions (see
Subsection \ref{Models}).

Our technique to prove the LDP for $Z_n/n$ is to
consider the recurrence
relation for $m_n(\lambda) = E[\exp\{\lambda Z_n\}]$ obtained from
(\ref{Z-def}):
\begin{equation}\label{recursion}
 m_{n+1}(\lambda) \ = \ ({1-e^\lambda})
\frac{m'_n(\lambda)}{s_n} + e^\lambda m_n(\lambda).\end{equation}
Dividing through by $m_n(\lambda)$, we write
\begin{equation}\label{recursion_divided_through}
\frac{m_{n+1}(\lambda)}{m_n(\lambda)} \ = \
\frac{1-e^\lambda}{s_n/n} \frac{m'_n(\lambda)}{nm_n(\lambda)} +
e^\lambda.\end{equation}
The idea now is to take the limit on $n$ in
the above display.  When the ``pressure" $\Lambda$ exists, it satisfies
$\Lambda(\lambda)  =  \lim_{n\rightarrow \infty} (1/n)\log
m_n(\lambda)$.  In this case, it is natural to suppose that the
limits
\begin{eqnarray}
\Lambda'(\lambda) &=& \lim_{n\rightarrow \infty}
\frac{m_n' %W David's correction
(\lambda)}{nm_n(\lambda)}
\label{pressure_derivative_limit}
\\
e^{\Lambda(\lambda)} & = & \lim_{n\rightarrow\infty}
\frac{m_{n+1}(\lambda)}{m_n(\lambda)}\label{exponential_limit}
\end{eqnarray}
 both exist.  Then,
from (\ref{recursion_divided_through}), we can write the ODE
\begin{equation}\label{ODE}e^{\Lambda(\lambda)} \ = \ \frac{1-e^\lambda}{\alpha}\Lambda'(\lambda) +
e^\lambda;
\ \ \ \La(0)=0.\end{equation} One can compute the solution of this
differential equation (cf. (\ref{Pressure_2})) and show it is
unique.

The main task is to show that the pressure and limits (\ref{pressure_derivative_limit}) and
(\ref{exponential_limit}) exist. But, the pressure exists as a
consequence of the path LDP for $Z_{\lfloor nt\rfloor}/n$ by contraction
principle. We note, in principle, one can try to compute the
pressure or the rate function from (\ref{I Path I})
 by the calculus
of variations, but we found it difficult to solve the associated
Euler equations (cf. near \eqref{PL}).

Finally, we show (\ref{pressure_derivative_limit}) and
(\ref{exponential_limit}) exist by extending $m_n(\lambda)$ to the
complex plane, and then analyzing its zeroes and analytic
properties. These estimates are also useful for the central limit
theorem arguments.

We now mention a different approach, in the spirit of
\cite{Flajolet-Gabarro-Pekari-05}, when $s_n$ is linear with slope
$\alpha =2$ and also interestingly $\alpha =1/2,1$, to compute the
pressure from analysis of the generating function $G(\lambda,z)=
\sum_{n\geq 1}m_n(\lambda)z^{n-1}$.  From (\ref{recursion}), we can
write the linear PDE
\begin{equation}\label{generating_function_pde}\frac{\partial
G}{\partial z}(1-e^\lambda z) + \frac{e^\lambda
-1}{\alpha}\frac{\partial G}{\partial \lambda} \ = \ e^\lambda G
.\end{equation} One can solve implicitly this PDE, and locate at
least heuristically a singular point.  Then, formally, from root
test asymptotics, the pressure would be the reciprocal of the
location of the singularity.

The difficulty is in
%S reworded -- maybe one doesn't identify the analyticity but establish it
establishing the analyticity of the solution and identifying
its singularity.  For urn models of ``subtraction type'' where the
mass added is an integer, and which satisfies a balance condition,
\cite{Flajolet-Gabarro-Pekari-05} uses this program to obtain large deviations
and the CLT.
 However, the cases $s_n$ is linear with slope
$\alpha=1/2,1,2$, and more generally the urns associated with
non-integer $s_n$
are not covered by their arguments which seem to rely on
integer additions with a certain ``negative" structure. On the other
hand, we are able to supply the needed analyticity and singularity
identification when $s_n$ has
slopes $\alpha =1/2,1,2$, and in
this way prove
%S we don't give another proof in alpha =1,1/2 provide another proof of
the
LDP for $Z_n/n$ (Theorem \ref{T3}) in these cases.

The plan of the paper is to state the results in Subsection
\ref{Results}, discuss applications to random graphs in Subsection
\ref{Models}, give the generating function proof of Theorem \ref{T3}
in Section \ref{PT1}, prove the path LDP (Theorem \ref{T2}) in
Section \ref{PT2}, give the ODE-method proof of Theorem \ref{T3} and
prove the LLN and CLT (Theorem \ref{T4}) in Section
\ref{ODE_section}, and conclude in Section \ref{Concludingremarks}.

\subsection{Results}
\label{Results} We recall the setting for large deviations.
% \cite{Varadhan(1966)}.
A sequence $\{X_n\}$ of random variables with values in a separable
complete metric space $\sX$ satisfies the large deviation principle
(LDP) with speed $n$ and rate function $I:\sX\to[0,\infty]$ if $I$
has compact level sets $\{\bx:I(\bx)\leq a\}$, and for every Borel
set $U\in\B_\sX$, \begin{eqnarray}
  - \inf_{\bx\in
U^\circ} I(\bx) &\leq &\liminf_{n\to\infty}\frac1n\log\Pr(  X_n\in
U)\nonumber\\
&\leq&  \limsup_{n\to\infty}\frac1n\log\Pr(  X_n\in  U)  \ \leq \ -
\inf_{\bx\in \bar U} I(\bx). \label{LDP_two_bounds}
\end{eqnarray}
[Here $U^\circ$ is the interior of $U$ and $\bar{U}$ is the closure
of $U$.]

Often the rate function is given in terms of the Legendre transform
of the ``pressure" when it exists.  When $\sX=\RR$, this
representation takes the form
\begin{equation}
  \label{Def I}
  I(x)\ =\ \sup_{\la\in\RR}\big\{ \la x-\log \La(\la)\big\},
\end{equation}
where we recall the pressure satisfies
\begin{equation}
\label{1D Pressure}
\La(\la)\ :=\ \lim_{n\to\infty}\frac{1}{n}\log E[e^{\la Z_n}].
\end{equation}

%%%%%%% Random coeff setting

Recall now the Markov chain $Z_n$ (\ref{Z-def}) corresponding to sequence $\{s_n\}$
and parameter
$1< \alpha<\infty$ (\ref{s_n}).

\begin{theorem}\label{T3}
The sequence $Z_n/n$ satisfies LDP with speed $n$ and good strictly
convex rate function
 $I$ given by \eqref{Def I} with pressure
\begin{equation}
\La(\la)\ =\  - \log \left( \frac{\alpha}{e^\la-1}
\int_0^\lambda\left(\frac{e^s -1}{e^\la-1}\right)^{\alpha -1}
ds\right) \ \ \ \  \ {\rm for \ }\lambda \neq
0\label{Pressure_2}\end{equation} and $\Lambda(0)= 0$.
\end{theorem}

\begin{remark}
\label{rmk1} \rm We note, when
$s_n=\alpha n$ with
$\alpha
=1/2,1$ and $k_0\leq s_1$, the LDP is also found, with pressure given by formula
(\ref{Pressure_2}), using the generating function approach (cf.
Section \ref{PT1}).
%S added alpha =1/2
For $\alpha =1/2$ and integer $\alpha\geq 1$, the integral in \eqref{Pressure_2} can
be evaluated explicitly.
\end{remark}

We now consider the LDP for the family of stochastic
processes $\{X_n(t): 0\leq t\leq 1\}$ obtained by linear
interpolation of the Markov chain \eqref{Z-def},
$$X_n(t)\ :=\ \frac{1}{n} Z_{\lfloor nt\rfloor-k_0+1}+\frac{nt-\lfloor nt\rfloor}{n}(Z_{\lfloor nt\rfloor-k_0+2}- Z_{\lfloor nt\rfloor-k_0+1}) \ \ {\rm for \ } t\geq \frac{k_0}{n}$$
%S put for in display and removed with in favor of word and
and
 $X_n(t):=t$ for $0\leq t\leq k_0/n$.
The trajectories of $X_n(t)$ are non-decreasing Lipschitz functions with constant at most $1$.
\begin{theorem}
\label{T2}
As a sequence of $ C([0,1];\RR) $-valued random variables, $X_n$ satisfies the LDP with the rate function
$I: C([0,1];\RR)\to [0,\infty]$ given by
\begin{equation}\label{I Path I}
  I(\varphi)\ =\ \int_0^1\left[\dot\varphi(t)\log\frac{\alpha t \dot\varphi(t)}{\alpha t-\varphi(t)}
  + (1-\dot\varphi(t))\log\frac{\alpha t (1-\dot\varphi(t))
  }{\varphi(t)}\right]\,dt
\end{equation}
if $\varphi(t)$ is differentiable for almost all $t$,  $\varphi(0)=0$, $0\leq\dot\varphi \leq1$, and
the integral converges; otherwise, $I(\varphi)=\infty$.
\end{theorem}

By the contraction principle,  Theorem \ref{T2} implies the LDP for
$Z_n/n$ with the rate function given by the variational expression
\begin{equation}I(x) \ = \ \inf\Big\{I(\varphi): \; \varphi(0)=0,\; \varphi(1)=x\Big\}.
  \label{contracted_rate}
  \end{equation}
In general, optimal trajectories are not straight lines---exceptions are the LLN trajectory
$\varphi_\alpha(t) = t \alpha/(\alpha +1) $
and the extreme case $\varphi(t) = t$---but they try to stay near the LLN line (for which $I(\varphi_\alpha)=0$)
to minimize cost before going to destination $x$ (cf. Fig. \ref{Fig5}).

\begin{figure}[hbt]\begin{center}
\includegraphics[height=1.5in]{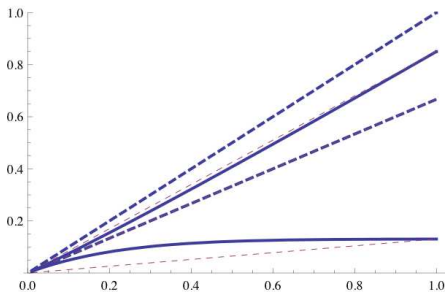} % change .8 to .85 in the picture -- it looks slightly better in Mathematica, but no perceptible difference here
\end{center}
\caption{\label{Fig5}Thick curves are numerical solutions of the Euler equations for \eqref{contracted_rate} with $\alpha=2$ for  $x=0.13, 2/3, 0.85, 1$. Dashed lines are straight lines from $(0,0)$ to  $(1,x)$.}
\end{figure}

Lemmas for the proof of Theorem \ref{T3}
 give Normal
 %SS capitalized normal (to make it abnormal (just kidding))
approximation.  The law of large numbers also follows
from Theorem \ref{T3}.
\begin{theorem}\label{T4}
We have
\begin{eqnarray*}
\frac{Z_n}{n}&  \xrightarrow{\rm a.s.}&
\frac{\alpha}{\alpha+1}
\end{eqnarray*}
and also
%SS squared variance in N(0,\sigma^2) below
\begin{eqnarray*}
\frac{1}{\sqrt{n}}(Z_n-E[ Z_n])&\xrightarrow{\rm\; D\;}&
N(0,\sigma^2)
\ \ \ {\rm where \ \ }\sigma^2\ =\
\frac{\alpha^2}{(1+\alpha)^2(2+\alpha)}.
\end{eqnarray*}
\end{theorem}
\begin{remark}
\label{rmk2} \rm The conclusions of Theorem \ref{T4} also hold
when $s_n=\alpha n$ with $\alpha=1/2, 1$, and $k_0\leq s_1$ (cf. Subsection \ref{proof_of_T4}).\end{remark}

\subsection{Applications to random graph models}
\label{Models}

With respect to random graphs, the Markov chain $Z_n$, representing
the number of leaves, sits in the intersection of at least two
models, that is preferential attachment graphs with
linear-type weights, and uniformly and planar oriented trees.  Also, when $s_n = \alpha n$ for $\alpha =1/2$,
$Z_n$ represents the count of cherries in Yule trees.

\subsubsection{Preferential attachment graphs}  Preferential attachment
graphs have a long history dating back to Yule (cf. \cite{Mitzenmacher}).
  However, since the
work of Barabasi-Albert
\cite{Albert-Barabasi-02,Albert-Barabasi-99}, these graphs have been
of recent interest with respect to modeling of various ``real-world"
networks such as the internet (WWW), and social and biological
communities.  Leaves, or nodes with degree one,
 in these
networks of course represent sites with one link, or members at the
``fringe." [See books \cite{Durrett-06}, \cite{Chung-Lu},
\cite{Bonato} for more discussion.]

The idea is to start with an initial connected graph $G_1$ with a
finite number of vertices, and say no self-loops
%added parenthetical remark--this is why we say no self loops, right?
(so that the vertices have well-defined degrees).
At step $1$, add
another vertex, and connect it
to a vertex $x$ of $G_1$ preferentially,
that is with probability proportional to its weight,
%S G_0 to G_1
$f(d_x)/\sum_{y\in G_1}f(d_y)$, to form a new graph
%WW removed "is" is
$G_2$.
Continue in this way by adding a new vertex and connecting it
preferentially to form $G_k$ for $k\geq 1$. Here, the ``weight" of a
vertex is a function $f$  of its degree $d_x$.  When $f:\mathbb N \rightarrow
\R_+$ is increasing, already well connected vertices tend to become
better connected, a sort of reinforcing effect.

Our results are applicable to
linear weights, $f(k)= k+\beta$, for $\beta>-1$,
%S added word mean below
which correspond to certain ``power law" mean degree sequences.  Namely,
let $Z_k(n)$ be the number of vertices in $G_n$ with degree $k$ for
$k\geq 1$.  It was shown, by martingale arguments, in
\cite{BRST-01}, \cite{Mori-01}, and by embedding into branching
processes, in \cite{Rudas-Toth-Valko-07} that $Z_k(n)/n\rightarrow
r_k$ a.s. where
%S denominator 1+\beta should be k+\beta
$$r_k \ =\
\frac{2+\beta}{k+\beta}\prod_{j=1}^{k}\frac{j+\beta}{j+2+2\beta} \
\sim\
1/k^{3+\beta}.$$ From the applications point of view, the
parameter $\beta$ can sometimes be matched to empirical network data
where similar power-law behavior is observed.

As the number of vertices with degree $1$, or the ``leaves" $Z_1(n)$,
increases by one in the next step when a non-leaf is selected, and remains
the same when a leaf is chosen, we see that $Z_1(n)$ corresponds to the
Markov chain $Z_n$ with $s_n$ as specified below.
%S reworked this paragraph to make correspondence more clear corresponds to is the number of vertices with degree
%$1$, or as mentioned before the ``leaves," $Z_1(n)$, in the graph
%$G_n$.
Since at each step the total degree of the graph augments by two,
the probability at step $n$ that a vertex $x\in G_n$ is
%S n should have been n-1 in second paren below
selected is $(d_x+\beta)/(d_{G_1}+ 2(n-1)+(n-1+|G_1|)\beta)$. Then,
$s_n=\big(d_{G_1} +2(n-1)+(n-1+|G_1|)\beta\big)/(1+\beta)$ and $\alpha =
(2+\beta)/(1+\beta)$.  Here $d_{G_1}$ and $|G_1|$ are the total
degree and number of vertices in $G_1$ respectively.

One can also ``randomize" the model by adding a random number of
edges at each step:
  Let
  $\{\gamma_{i}\}$ be a sequence of independent identically distributed random
variables on $\mathbb N$ with
%$p^L_1= P(L=1)$,
%S don't need $p^\gamma_1=P(\gamma_1=1)>0$ and
finite
mean,
%$\bar{L}=E[L]$,
$\bar{\gamma}=E[\gamma_1]<\infty$.
Then, at step $n$, we add a new vertex to the graph $G_n$ and connect it to
%$L_{n}$ vertices in $G_n$ selected
a node selected preferentially from $G_n$
% and say independently (some of these vertices may be the
%same)
with $\gamma_{n}$ edges put between them.
% new vertex and the $j$th selected vertex in $G_n$
%for $1\leq j\leq L_n$.
The effect of these random edge additions is to randomize further the weight given the nodes in the graph; the ``deterministic" model above is when $P(\gamma_1=1)\equiv 1$.
%S removed in favor of above since we don't define it p^\gamma_1\equiv 1$.
%$p^L_1,p^\gamma_1\equiv 1$.]
In \cite{athreya-2007},
by embedding into branching processes, it was shown, when
%$p^L_1\equiv
%1$ and
$E(\gamma_1\log\gamma_1)<\infty$, that $Z_k(n)/n
\stackrel{P}{\rightarrow} r_k$
%S added next phrase
where $r_k$ is given by an integral formula with asymptotics $r_k \sim
1/ k^{3+
  \beta/\bar{\gamma}}$.
%S maybe don't need [In \cite{athreya-2007}, $G_1$ was taken to be two
%vertices with a single edge, but one can generalize to an arbitrary
%finite graph $G_1$.]
See \cite{Cooper-Frieze} for an even more general
randomized model
where also a
%weak-type
LLN is proved.  Other models with different
random edge addition schemes are found in \cite{dereich-2008}, \cite{Mori-Zsolt}.

In this ``randomized" model, we now define the notion of ``generalized leaves" or ``buds,"
that is, those nodes which connect to exactly one other vertex,
albeit possibly with several edges linking them.
Leaves are those buds with exactly one edge connection.
Similar to the leaves in the deterministic setting, at step $n$, the bud count increases by one exactly when the new vertex,
a fortiori a bud,
connects to a non-bud in $G_n$,
and remains the same when the new vertex links to a bud in $G_n$.

Then, the Markov chain $Z_n$, representing the number of buds at step $n$, corresponds to
%a newly added vertex is a leaf at step $n$ exactly when it connects to a single node,
%$L_n =1$, and one edge is drawn between these vertices, $\gamma_n1=1$.  Hence, in this setting,
$s_1=(d_{G_1}+|G_1|\beta)/(1+\beta)$ and for $n\geq 2$, $$s_n \ =\ \frac{1}{1+\beta}\Big[d_{G_1}+
2\sum_{i=1}^{n-1}\gamma_{i} + \big(n-1+|G_1|\big)\beta\Big],$$
which satisfies assumption
\eqref{s_n}, by the
strong law of large numbers,
with $\alpha = (2\bar{\gamma}
+ \beta)/(1+\beta)$ a.s. In particular, by
our results, we obtain ``quenched" LLN, CLT and LDP's, with probability one with respect to $\{\gamma_{i}\}$,
for the buds in this scheme.

\subsubsection{Random recursive trees} Random recursive trees are
also well-established models, dating to the 1960's, with
applications to data sorting and searching, pyramid schemes, spread
of epidemics, chemical polymerization, family trees (stemma) of
copies of ancient manuscripts etc. Leaves in these trees correspond
to ``shutouts" with respect to pyramid schemes, nodes with small
``affinity" in polymerization models, ``terminal copies" in stemma of
manuscripts etc. See \cite{Mahmoud-Smythe}, \cite{Szymanski} and
references therein (e.g. \cite{Najock-Heyde} etc.) for more
discussion.  Below, we mention connections with Stirling
permutations.

Similar to preferential attachment, the recursive schemes form a sequence
of trees.  One starts with a single vertex labeled $0$, and then
adds a vertex at step $n\geq 1$, labeled $n$, to one of the $n$
nodes already present. When the choice is made uniformly, the model
forms a {\it uniformly} recursive tree. However, when the choice is
made proportional to the degree, a non-uniform or {\it plane
oriented} tree is grown. The interpretation is that in uniform
trees, at each distance from the root $0$, there is no ordering of
the labels. However, in plane-oriented trees, different orders give
rise to distinct trees; that is, after selecting a parent node at
random at step $n$ with $k$ children with a certain cyclic order of
labels, there are $k+1$ locations where the new label $n$ can be
inserted.

Plane oriented trees are similar to preferential attachment graphs
with $\beta =0$ (cf. Chapter 4 \cite{Durrett-06}), and the number of
leaves corresponds to the Markov chain with $k_0=1$, $s_n = 2n-1$
for $n\geq 1$, and $\alpha =2$.  On the other hand, the number of
leaves in uniformly recursive trees corresponds to the case $k_0=1$,
$s_n = n$ for $n\geq 1$, and $\alpha =1$. With respect to both types
of recursive trees, LLN's and CLT's have also been proved by
combinatorial, urn and martingale methods (see
\cite{Mahmoud-Smythe}, \cite{Janson04}). So, in this context, our
results give alternate proofs of the LLN and CLT for these recursive trees, and also an LDP for $Z_n/n$.
We note also Theorem \ref{T2} yields a path LDP for planar
oriented trees.

We comment now on recent connections of planar oriented trees with
Stirling permutations (cf. \cite{janson-2008},
\cite{janson-kuba-panholzer-2008}). A Stirling permutation of length
$2n$ is a permutation of the multiset $\{1,1,2,2,\ldots, n,n\}$ such
that for each $i\leq n$ the elements occurring between the two $i$'s
are larger than $i$ (cf. \cite{Gessel-Stanley}).  It turns out that
each permutation is a distinct code for a plane oriented recursive
tree with $n+1$ vertices.

Quoting from \cite{janson-2008}, consider the depth first walk which
starts at the root, and goes first to the leftmost daughter of the
root, explores that branch (recursively, using the same rules),
returns to the root, and continues to the next daughter and so on.
Each edge is passed twice in the walk, once in each direction. Label
the edges in the tree according to the order in which they were
added--edge $j$ is added at step $j$ and connects vertex $j$ to an
previously labeled vertex. The plane recursive tree is coded by the sequence of
labels passed by the depth first walk.  With respect to a tree with
$n+1$ vertices, the code is of length $2n$, where each of the labels
$1,2, \ldots, n$ appears twice.  Adding a new vertex means inserting
the pair $(n+1)(n+1)$ in the code in one of the $2n+1$ places.

In a Stirling permutation $a_1a_2\cdots a_{2n}$, the index $0\leq
i\leq 2n$ is a plateau if $a_i=a_{i+1}$ (where $a_{2n+1}=0$).
\cite{janson-2008} shows that the number of leaves in a plane
oriented tree with $n+1$ vertices is the number of plateaux in a
random Stirling permutation. See \cite{janson-2008} for more
details.

\subsubsection{Yule trees}
Since Yule's influential 1924 paper \cite{Yule-1924}, Yule trees, among other processes, have been used widely to model phylogenetic
evolutionary relationships between species (see \cite{Aldous-2001} for an interesting essay).  In particular, the counts of various shapes and features of these trees can be studied, and matched to empirical data to test evolutionary hypotheses.  In \cite{McKenzie-Steel-00}, a LLN and CLT is proved for the number of cherries, or pairs of leaves with a common parent, in Yule trees.  Associated confidence intervals are computed, and some empirical data sets are examined to see their compatibility with ``Yule tree" genealogies.  Other related statistical tests and limit results can be found in
\cite{Blum-Francois-2005,Blum-Francois-Janson,Rosenberg-2006}.

In the Yule tree process, one starts with a root vertex.  It will split into two daughter nodes at step $2$, each of which is equally likely to split into two children at step $3$.  At step $n$, one of the $n$ leaves in the tree is chosen at random, and it then splits into two daughters, and so on.  It is easy to see that the number of cherries at step $n$ is given by the Markov chain $Z_n$ with $s_n = n/2$ and $k_0 = 0$.  Correspondingly, our results give different proofs of the LLN and CLT for the cherry counts, as well as an LDP for $Z_n/n$.

\section{Proof of Theorem \ref{T3}: Cases $s_n = \alpha n$, $\alpha =1/2,1,2$}
\label{PT1}

In this section we give a
proof of the LDP in Theorem
\ref{T3}, via an analysis of a generating function, for the
special case $s_n=\alpha n$ for $\alpha=2$,
%S took out which is
at the
intersection of several models, and two additional values $\alpha=1/2, 1$ not covered
by Theorem \ref{T2}.

%%%%%%%%%%%%%%
%% Complex analysis
%%%%%%%%%%%%%%%
The LDP is obtained by G\"artner-Ellis theorem \cite[Theorem 2.3.6]{Dembo-Zeitouni-98}
%S should it be Gartner Ellis? %WW YES, thank you!!!
 from the following.
\begin{proposition}\label{P 1D} Suppose $s_n=\alpha n$  with $\alpha=1/2$, $1$, or $2$,
and
 $Z_1=k_0\leq s_1$. The limit \eqref{1D Pressure} exists, and is given by the smooth function
 \begin{eqnarray}
  \La(\la)& =&  \left\{\begin{array}{ll}
\log \left(\frac{\sqrt{e^{\lambda }-1}}{\arctan
  \left(\sqrt{e^{\lambda }-1}\right)}\right)& \mbox{\rm  if } \la>0  \\
   \log \left(\frac{ \sqrt{1-e^{\lambda }}}{{\rm arctanh}\left(\sqrt{1-e^{\lambda
   }}\right)}\right) & \mbox{\rm  if }\la<0\ \ \ \ \ \ \ {\rm for \ }\alpha = 1/2\end{array}\right.\nonumber\\
      &=&\ \ \ \ \log \frac{e^\la-1}{\la}\ \ \ \ \ \ \ \ \ \ \ \ \ \ \ \ \ \ \ \ \ \ \ \ \ \ \ \ \ \ \ \ \,  \ {\rm for \ }\alpha =1\nonumber\\
    &=&\ \ \ \ \log\left(\frac{\left( e^{\lambda } -1 \right)^2}{2\,\left( e^{\lambda } -1  - \lambda  \right) }\right)
   \ \ \ \ \  \ \ \ \ \ \ \ \ \ \ \ \ \ \  {\rm for \ }\alpha =2
 \label{R-def}\end{eqnarray} when $\lambda \neq 0$, and $\Lambda(0)=0$.
\end{proposition}
We  follow analytic arguments adapted
 from \cite{Flajolet-Gabarro-Pekari-05}.
Since $0\leq k_0\leq Z_n\leq n+k_0-1$,  we get
$0< m_n(\la)\leq
e^{(n+k_0)\la^+}$ with $\la^+=\max\{\la,0\}$. Therefore,   for all
complex $z$ with $|z|<e^{-\la^+}$, $G(z,\la)$ is well defined and
satisfies \eqref{generating_function_pde} with the initial condition
$G(0,\la)=e^{k_0\la}$. The coefficients of this PDE do not vanish in
the regions %W |Z|<exp-la (David's correction)
$$\mathcal{D}_+=\big\{(z,\la): \la>0, |z|<e^{-\la}\big\}, \; \mathcal{D}_-=\big\{(z,\la): \la<0, |z|<1\big\}.$$
For $\la\ne 0$, the PDE can be solved by the method of characteristics.
Clearly, $m_n(\la)$ are the coefficients in the Taylor expansion  at $z=0$ of the solution, and $e^{-\La(\la)}$ is the radius of convergence of the series that can be determined by singularity analysis.

For $\alpha=1$, $k_0=0,1$;
using initial condition $G(0,\la)=e^{k_0\la}$ we get
$$G(\lambda, z) \ =\ e^{k_0(z(e^\lambda -1)-\lambda)}\frac{e^\lambda -1} {1-\exp(z(e^\lambda -1)-\lambda)}.$$
Hence, the singularity of $G$ as a function of $z$ nearest to the
origin is a simple pole at
$$z_0\ =\ \frac{\lambda}{e^\lambda-1}.$$
By Darboux's asymptotic method  \cite[Ch. 8]{Olver-74},
$$\frac1n\log E(\exp (\la Z_n))\ \to\  \log \frac{e^\lambda-1}{\lambda}$$
and the LDP follows.

%%%%%%%%%%%%%%%%
%%% alpha=1/2
%%%%%%%%%%%%%%%%
Next, consider $\alpha=1/2$. In this case,
the solutions of the PDE   differ  depending on the region $\mathcal{D}_\pm$,
but  are explicit so there are no difficulties in constructing
%S it was with construction changed to in constructing
 their analytic extensions. Using
the  initial condition $k_0=0$, we have  $G(0,\la)=1$, and the solution of
\eqref{generating_function_pde} is
$$
G(\lambda, z) \ =\
\begin{cases}
\displaystyle\frac{\sqrt{e^\la-1}}{\tan\big(\arctan\sqrt{e^\la-1}-z \sqrt{e^\la-1}\big) }& \la>0 \\
\displaystyle\frac{\sqrt{1-e^\la}}{\tanh\big({\rm arctanh}\sqrt{1-e^\la}-z \sqrt{1-e^\la}\big)}& \la<0
\end{cases}
$$
Hence, the singularity of $G$ as a function of $z$ nearest to the
origin is a simple pole at
$$z_0\ =\ \begin{cases}
\displaystyle\frac{\arctan\sqrt{e^\la-1}}{\sqrt{e^\la-1}} & \la>0\\
\displaystyle\frac{{\rm arctanh}\sqrt{1-e^\la}}{\sqrt{1-e^\la}} &\la<0
\end{cases}$$
Once again, by Darboux's asymptotic method  \cite[Ch. 8]{Olver-74},
$$\frac1n\log E(\exp (\la Z_n))\ \to\  \log 1/z_0$$
and the LDP follows.

%%%%%%%%%%%%%%%%
%%% end of proof for alpha=1/2
%%%%%%%%%%%%%%%%

For $\alpha=2$, the method of characteristics gives the following answer.
\begin{lemma}\label{LC1}
Suppose $\varphi$ is a function of one complex variable, analytic in
a domain $\mathcal{D}$ containing $(-\infty,-2)$. Then
 \begin{equation}\label{tilde G}
   {G}(z,\la)\ :=\
     (e^\la-1)^2\varphi(z (e^\la -1)^2+ 2\la - 2 e^\la).
      \end{equation}
      satisfies the PDE
      \eqref{generating_function_pde}
      for all $\la\ne 0$ and $|z|$ small enough.
      Furthermore, the initial condition is fulfilled at  $\la\ne 0$ if
 \begin{equation}  \label{PDE-INI}
 \varphi(2\la-2e^\la)\ =\ \frac{e^{k_0\la}}{(1-e^\la)^2}.
 \end{equation}
\end{lemma}
\begin{proof}
The verification of the initial condition is trivial, and the verification of the PDE is a straightforward calculation.
 Denoting for conciseness
$\psi(z,\la)=\varphi(z (e^\la -1)^2+ 2\la - 2 e^\la)$,
$\psi^{(1)}(z,\la)=\varphi'(z (e^\la -1)^2+ 2\la - 2 e^\la)$
we verify, for $\la\ne 0$ and $z$ such that $z (e^\la -1)^2+ 2\la - 2 e^\la\in\mathcal{D}$, that
$$
\frac{\partial {G}(z,\la)}{\partial \la}
%=2 (e^\la -1)e^\la \psi(z,\la)+
%(e^\la -1)^2 \left(2 z (e^\la -1)e^\la+ 2 - 2 e^\la\right)\psi^{(1)}(z,\la)$$
%$$
\ =\ 2 (e^\la -1)e^\la \psi(z,\la)+2(e^\la -1)^3 (ze^\la -1)\psi^{(1)}(z,\la)
$$
and $$
  \frac{\partial {G}(z,\la)}{\partial z}\ =\ (e^\la -1)^4 \psi^{(1)}(z,\la).
$$
Equation \eqref{generating_function_pde} now  follows by a calculation.
\end{proof}

Our next goal is to show that one can find a solution of \eqref{generating_function_pde} which can be analytically
extended in variable $z$ to a large enough domain. To this end, we need to analyze
\eqref{PDE-INI}   with  complex argument. The basic plan consists of noting that
function $f(x)=2(x-e^x)$ is analytic, strictly decreasing for $x>0$, strictly increasing for
$x<0$, and $f(0)=-2$.
 The derivative $f'(x)$  vanishes only at $x= 0$, so $f\left|_{[0,\infty)}\right.$ and $f\left|_{(-\infty,0]}\right.$ have   continuous inverses
$h_+: (-\infty,-2]\to[0,\infty)$ and $h_-:(-\infty,-2]\to(-\infty,0]$, and both are analytic on
$(-\infty,-2)$.

Clearly,  if we define
\begin{equation}\label{h2phi}
\varphi_\pm(\la)\ =\ \frac{e^{k_0h_\pm(\la)}}{(e^{h_\pm(\la)}-1)^2}\;,
\end{equation}
then $\varphi_+$ satisfies \eqref{PDE-INI} for $\la>0$, and $\varphi_-$ satisfies  \eqref{PDE-INI} for $\la<0$.
The goal is therefore to find the appropriate analytic extensions
of the functions
$h_\pm$.
%%%% New Proof of extension

\subsection{Construction of an analytic extension}

We need to analyze $f(z) = 2(z - e^z)$. The closely related function $z+e^z$ appears in  \cite[page 116]{Kober} but proofs are not included there; we give details for $f(z) = 2(z - e^z)$ for completeness.

%W added lemma
The proof relies on the following univalence criterion. % \cite[Theorem 2.16 on page 47]{Duren-83}
\begin{lemma}[Wolff-Warschawski-Nishiro]
\label{Warschawski}
If  $g$ is holomorphic in a convex region $\Omega$ and
$g'(\Omega)\subset H$, a half-plane with $0\in\partial H$, then $g$ is one-to-one on $\Omega$.
\end{lemma}
\begin{proof} This is  \cite[Theorem 2.16 on page 47]{Duren-83}  applied to the function $e^{i\theta}g(z)$ with  a real constant $\theta$ chosen appropriately  to rotate $H$.
\end{proof}
\begin{lemma}\label{LCC}
$f$ is a one-to-one mapping of the half-closed strip $\Sigma = \{z: 0< \Im (z) \leq \pi\}$ onto a slit closed half-plane: $\{w:\Im (w)\leq \pi\}\setminus (-\infty,-2]$. The boundary correspondence is as follows: $f$ maps $\RR + \pi i$ injectively onto $\RR + 4\pi i$, and $f$ is one-to-one on both $[0,+\infty)$ and $(-\infty,0]$ and maps each onto $(-\infty,-2]$ (cf. Fig. \ref{Fig1}).
\end{lemma}

\begin{proof}
We write $f(z)=g(\zeta)$, where  $\zeta=e^z$, $g(\zeta)=2(\log\zeta -\zeta)$ and $\log$ denotes the principal branch of the logarithm.  The function $\zeta=e^z$ is a one-to-one mapping on the strip $\Sigma$. The image of the interior of $\Sigma$ is the upper half-plane $\HH = \{\zeta: \Im(\zeta)>0\}$. Furthermore, the upper edge $\RR + \pi i$ of $\Sigma$ is mapped onto $(-\infty,0)$ and the bottom edge $\RR$ is carried onto $(0,+\infty)$.

The derivative of $g(\zeta)$ is %W David's correction
$g'(\zeta)=2\bar{\zeta}/|\zeta|^2-2$.
In particular, %W David's correction
$\Im(g'(\zeta))=-2\Im (\zeta)/|\zeta|^2<0$ for $\zeta \in \HH$. By
Lemma \ref{Warschawski} with $\Omega=\HH$,
% the Wolf-Warschawski-Nishiro  lemma \cite[Theorem 2.16 on page 47]{Duren-83} applied to  $ig(\zeta)$ on $\HH$,
$g$ is one-to-one on the half-plane $\HH$.  Under $g$ the image of $\HH$ is
$\{w:\Im (w)\leq 2\pi\}\setminus(-\infty, -2]$, and on the boundary we have $g(\RR^-)=\RR+2\pi i$ and $g(\RR^+)$ is the slit $(-\infty,-2]$ twice covered with $g(1)=-2$. Perhaps this latter statement is easier seen directly  from $f$;  the slit is covered twice: $f((-\infty,0])=f([0,+\infty))=(-\infty,-2]$.
\end{proof}

\begin{figure}[hbt]\begin{center}
\includegraphics[height=3in]{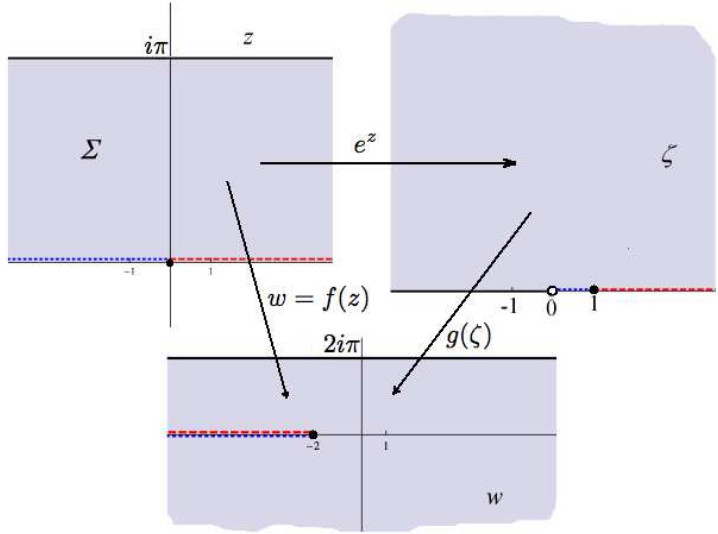}
\end{center}
\caption{\label{Fig1} Composition of  maps $f(z)=g(e^z)$ for the proof of Lemma \ref{LCC}.}
\end{figure}

%% temporary files to grab Fig1
%\begin{figure}[hbt]
%\begin{tabular}{c}
%\includegraphics[height=2.5in]{Fig1Mz}\hspace{1cm}
%\includegraphics[height=2.5in]{Fig1Mzeta}\\
%\includegraphics[height=1.5in]{Fig1Mw}
%\end{tabular}
%\caption{\label{Fig1} Composition of  maps $f(z)=g(e^z)$ for the proof of Lemma \ref{LCC}.}
%\end{figure}

%\subsection*{Symbols for pictures}

%$$\Sigma\;\; \Sigma_+ \;\; \Sigma_- \;\; \Delta_0 \;\; \Delta_1 \;\; \Delta_{-1} \gamma \;\; S\;\; S_\infty\;\; \HH \;\;\overline{\HH}\;\; $$

%$$ i\pi \;\; -i\pi \;\; 2i\pi\;\; -2i\pi\;\;$$

%$$ w=f(z)\;\; g(\zeta)\;\; e^z\;\; z \;\; w\;\; \zeta\;\; $$

We investigate the conformal mapping $f|\Sigma$ in more detail. The preimage of  $[-2,\infty)$ under $f|\Sigma$ is the curve $\gamma$ given by $\Re z=\log(\Im z/\sin \Im z)$, $0\leq \Im z < \pi$. This curve begins at the origin and becomes asymptotic to $\RR+ \pi i$ in the positive direction. By removing the curve $\gamma$, $\Sigma$ is cut into two parts, $\Sigma_-$ and $\Sigma_+$. $\Sigma_-$ is bounded by $\RR^-$, $\gamma$ and $\RR+\pi i$ with the latter line part of $\Sigma_-$. The region $\Sigma_+$ is bounded by $\gamma$ and $\RR^+$. Then $f|\Sigma_-$ is a conformal mapping of $\Sigma_-$ onto the half-closed strip $S = \{w: 0 < \Im w \leq 2\pi\}$ with $f(\RR + \pi i) = \RR +2\pi i$, $f(\RR^-) = (-\infty,-2)$ and $f(\gamma) = [-2,+\infty)$. Similarly, $f|\Sigma_+$ is a conformal mapping of $\Sigma_+$ onto the lower half-plane $\overline{\HH} = \{\bar z: z \in \HH \}$, where $\bar z$ denotes the complex conjugate of $z$, with $f(\RR^+) = (-\infty,-2)$ and $f(\gamma) = [-2,+\infty)$. (See Fig \ref{Fig3}.)
\begin{figure}[hbt]
\begin{center}
\begin{tabular}{c}
\includegraphics[height=1.5in]{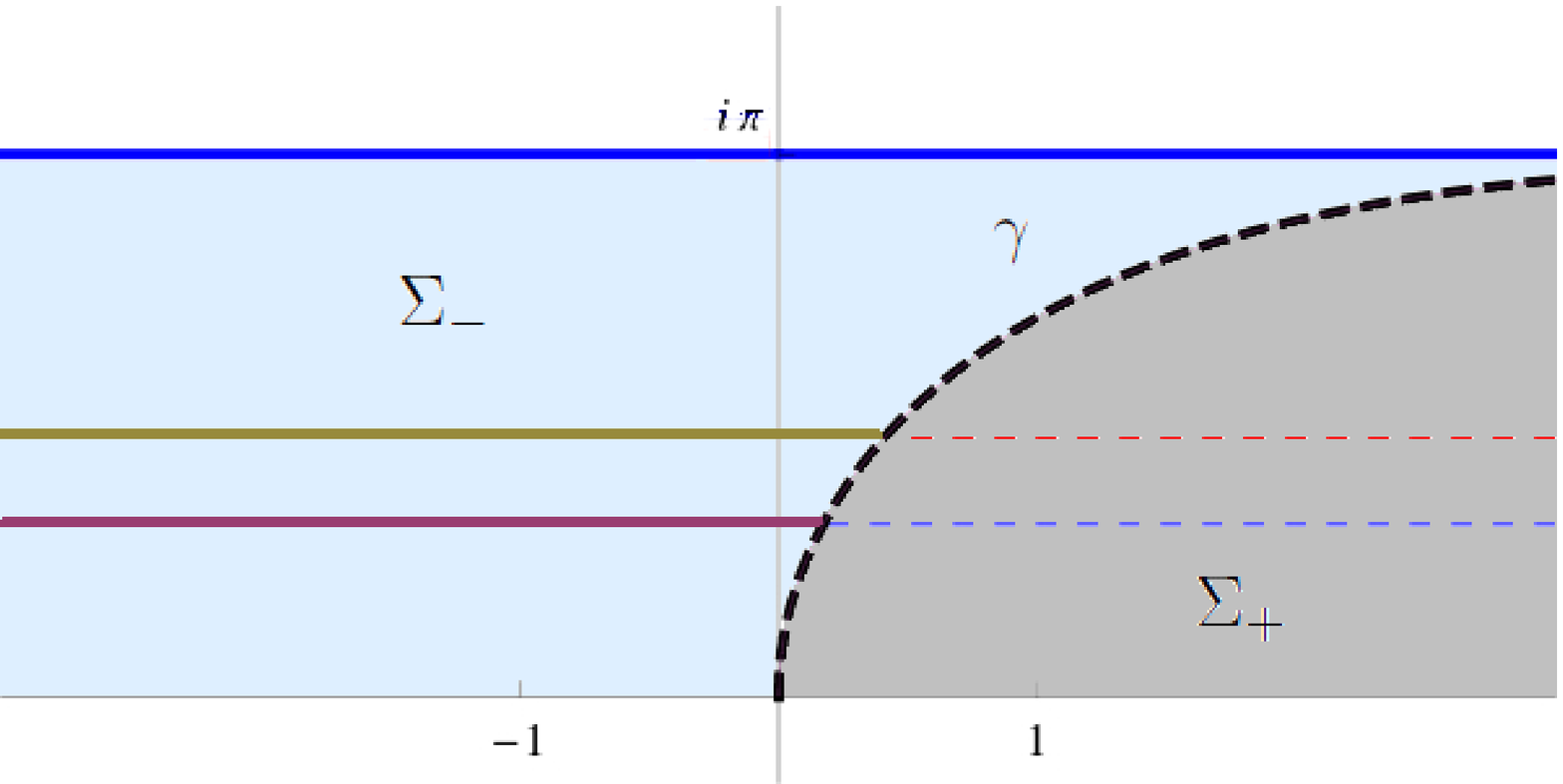}\\
{\Huge $\downarrow$}$^{f}$  \\
\includegraphics[height=1.5in]{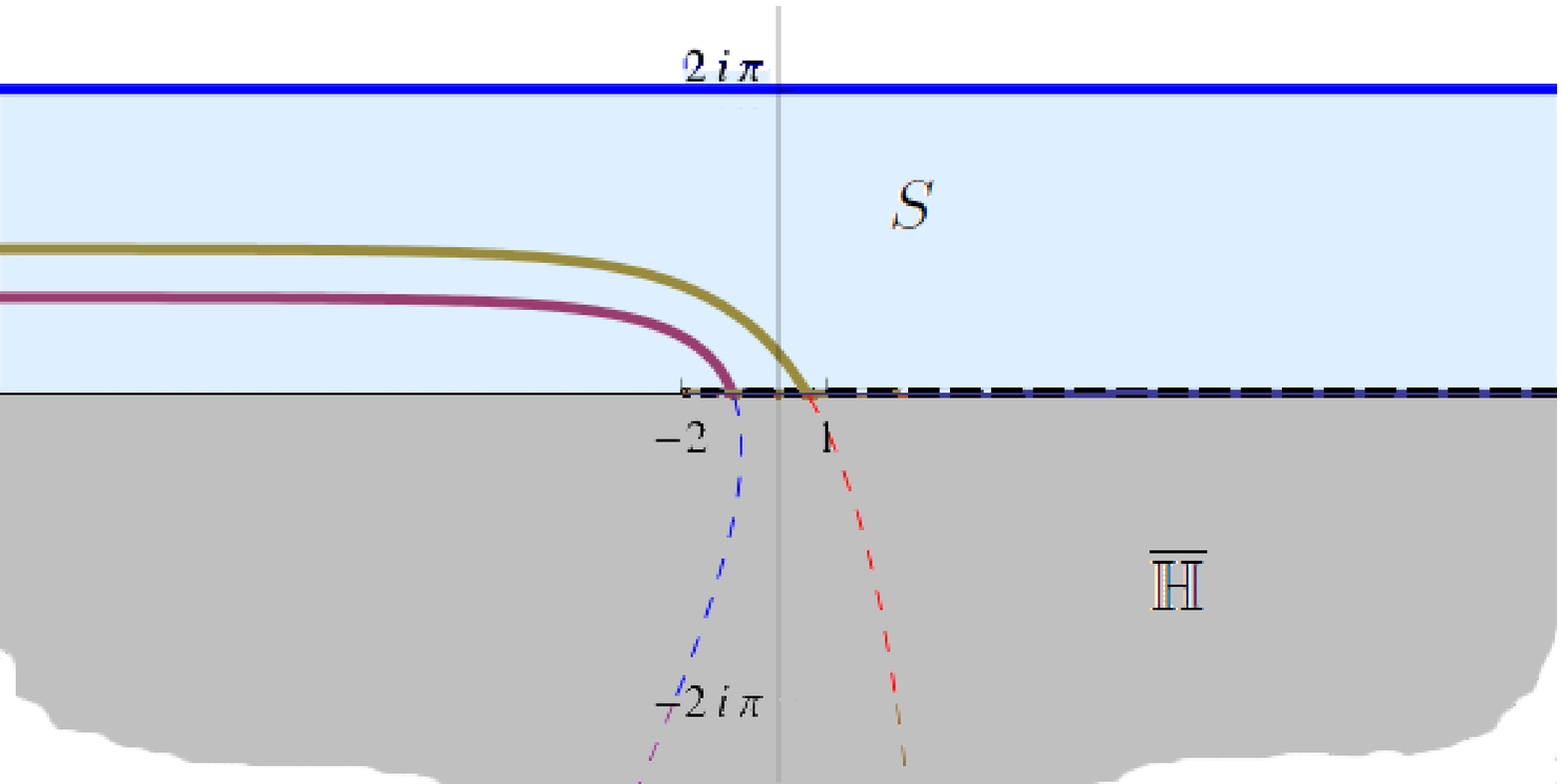}
\end{tabular}
\end{center}
\caption{\label{Fig3}  Conformal maps  $f|\Sigma_-$ and $f|\Sigma_+$.}
\end{figure}

Both maps $f|\Sigma_-$ and $f|\Sigma_+$ extend to conformal maps of larger regions. Because $f(\bar z) = \overline{f(z)}$, $f$ maps $\overline{\Sigma_+}$ conformally onto $\HH$. As $f(\RR^+)= (-\infty,-2)$, $f$ is a conformal map of $\Omega_+ := \Sigma_+\cup \RR^+ \cup \overline{\Sigma_+}$ onto the slit plane $\bar \HH \cup (-\infty,-2) \cup \HH = \CC \setminus[-2,+\infty)$. Let $h_+:\CC \setminus[-2,+\infty) \to \Omega_+$ be the inverse function for $f|\Omega_+$. The conformal extension of $f|\Sigma^-$ is more involved to describe. The fact that $f(\bar z) = \overline{f(z)}$ implies $f$ maps $\overline{\Sigma_-}$ conformally onto $\bar S$. Since $f(\RR^-) = (-\infty, -2)$, $f$ is a conformal map of $\Delta_0 = \Sigma_- \cup \RR^- \cup \overline{\Sigma_-}$ onto the slit closed strip $S_0 = S \cup (-\infty,-2)\cup \bar S = \{w: |\Im w | \leq 2\pi\} \setminus [-2,+\infty)$ with the upper (lower) edge of $\Delta_0$ corresponding to the upper (lower) edge of $S_0$. It is straightforward to verify that $f(z+2\pi i) = f(z) + 4\pi i$ for each $n \in \ZZ$. This functional relationship implies that $f$ is a conformal map of $\Delta_n = \Delta_0 + 2\pi n i$ onto $S_n = S_0 + 4\pi n i$ for each $n \in \ZZ$. Hence, $f$ is a conformal map of $\Omega_- := \cup_{n \in \ZZ} \Delta_n$ onto the infinitely slit plane $S_\infty := \bigcup_{n \in \ZZ} S_n = \CC \setminus \bigcup_{n\in \ZZ} \big([-2,+\infty)+4\pi ni\big)$. Let $h_-:S_\infty \to \Omega_-$ be the inverse function of $f|\Omega_-$. Because $\CC \setminus[-2,+\infty)  \supset S_\infty$, we may regard $h_+$ as defined on $S_\infty$, so $h_\pm$ have a common domain.
(See Fig. \ref{Fig4DM}.)
%(Include figures showing the conformal maps $f|\Omega_+$ and $f|\Omega_-$.)
\begin{figure}[hbt]
\begin{center}
\begin{tabular}{cc}
\includegraphics[height=1.2in]{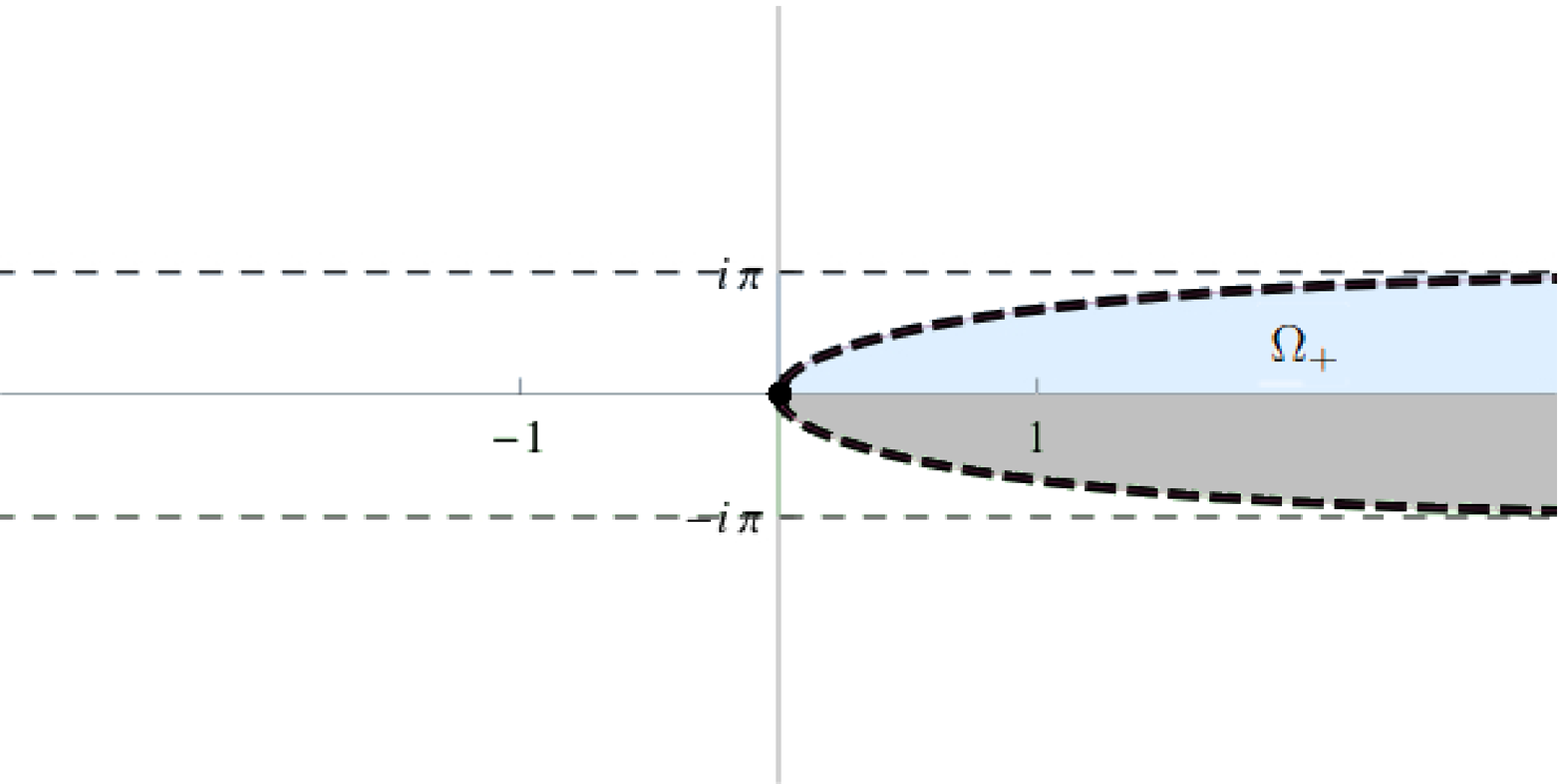}&\includegraphics[height=1.2in]{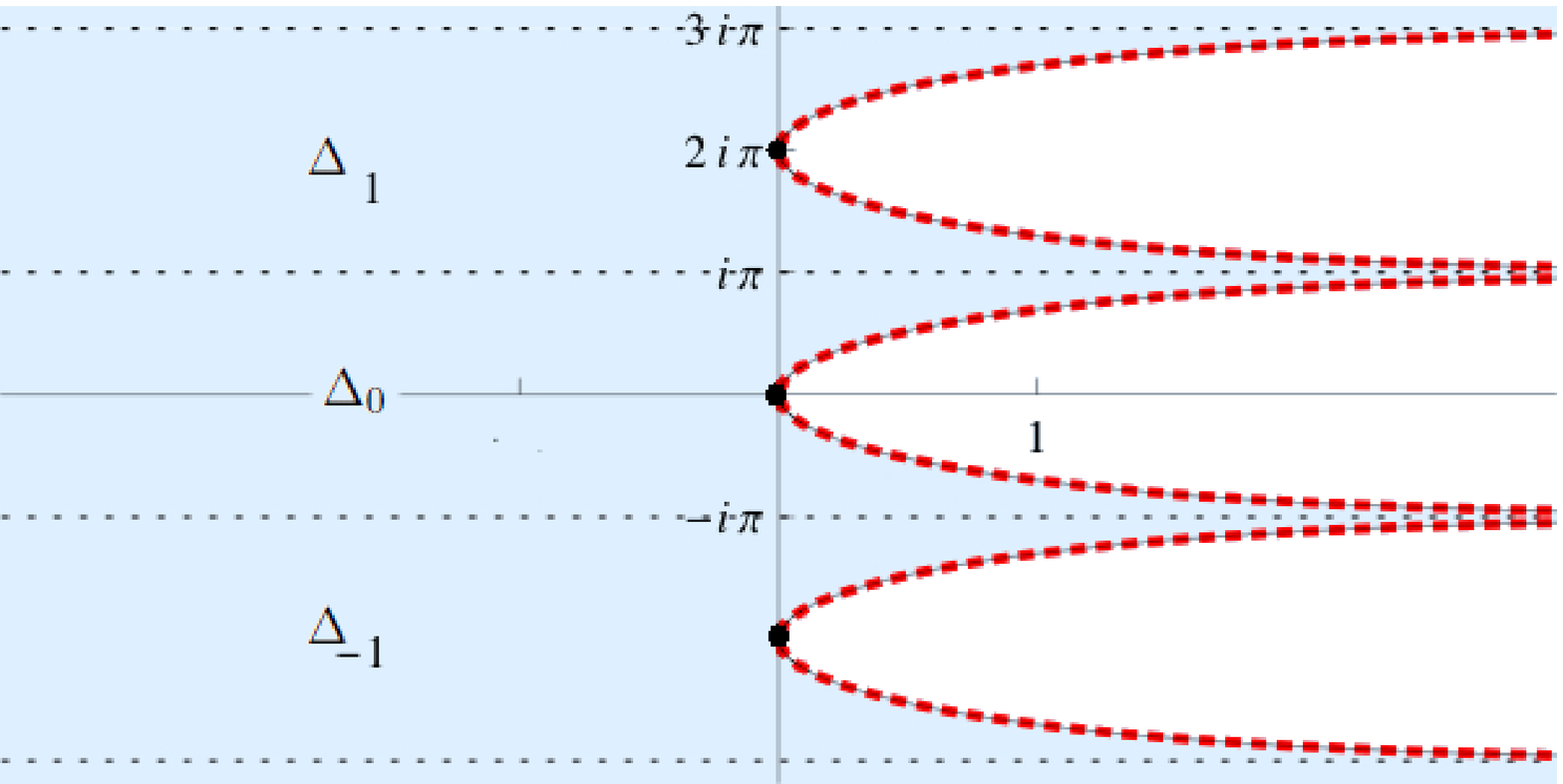}\\
{\Huge $\downarrow$}$^{f|\Omega_+}$& {\Huge $\downarrow$}$^{f|\Omega_-}$ \\
\includegraphics[height=1.2in]{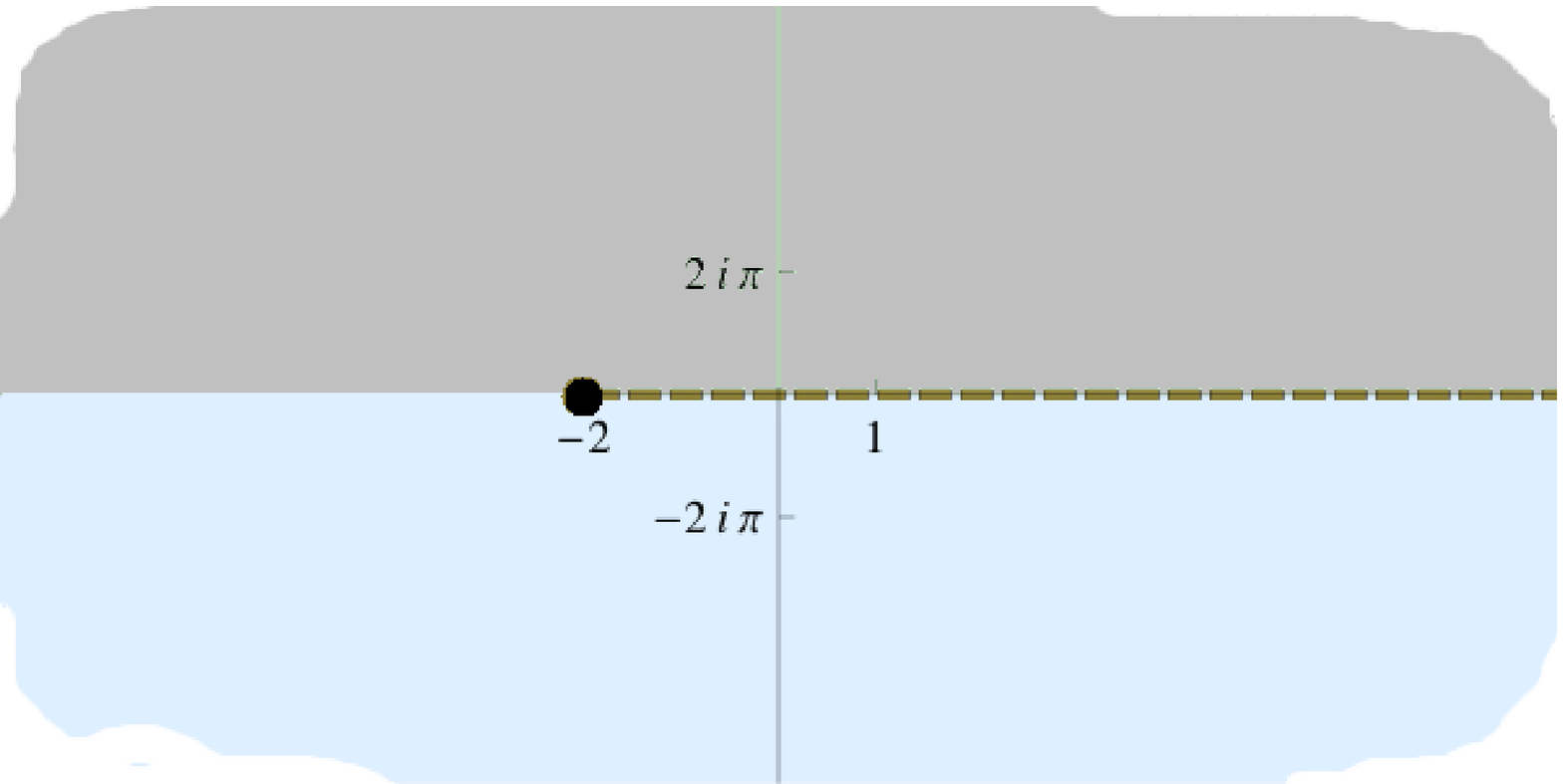}&\includegraphics[height=1.2in]{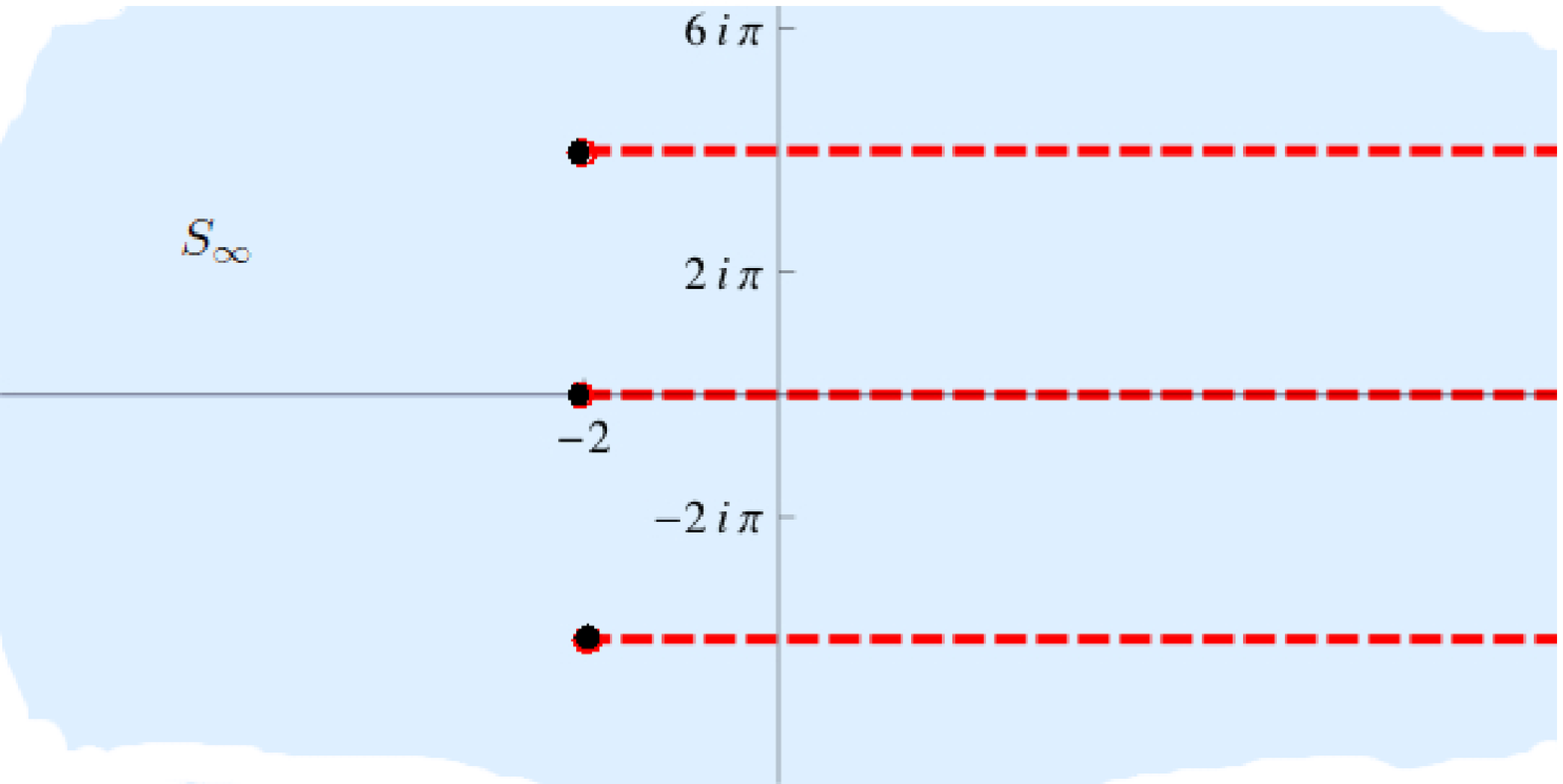}
\end{tabular}
\end{center}
\caption{\label{Fig4DM}  Conformal maps $f|\Omega_+$ and $f|\Omega_-$ with
 $\Omega_+ := \Sigma_+\cup \RR^+ \cup \overline{\Sigma_+}$ and
$\Omega_-=\bigcup_{n=-\infty}^\infty \Delta_n$.}
\end{figure}

The two conformal maps $h_\pm$ just constructed provide analytic extensions to $S_\infty$ of the real-valued functions $h_\pm$.
%%W  Back to old proof, slightly revised to include new notation
This allows us to define a pair of functions
 \begin{equation}\label{phi} %W serious correction caught by David
 \varphi_\pm(z):=\frac{e^{k_0 h_\pm(z)}}{(e^{h_\pm(z)}-1)^2}
 \end{equation}
which are analytic in $S_\infty$.

\begin{lemma}\label{L2}
For each point $u \in (-\infty,-2)$ the power series expansion of $\varphi_\pm$ has radius of convergence $2-u$,
 $\varphi_\pm$ is analytic in the slit disk $D(u,r)\setminus[-2,+\infty)$, where $r=r(u) = \sqrt{(u+2)^2 + 16\pi^2}$.
 Furthermore,
  $\varphi_\pm(w)\approx -1/(w+2)$ as $w\to -2$ in $S_\infty$. \end{lemma}
\begin{proof}
For each point $u \in (-\infty,-2)$ the power series expansion of $h_\pm$ has radius of convergence $|u+2|$, the distance from $u$ to $-2$ because $-2$ is the closest singularity of $h_\pm$. Also, $h_\pm$ is analytic in the slit disk $D(u,r)\setminus[-2,+\infty)$, where $r=r(u) = \sqrt{(u+2)^2 + 16\pi^2}$ is the distance from $u$ to $-2 \pm 4\pi i$.

Note that singularities of $\varphi_-$ that arise from $h_-(w)=2n \pi i$ are located
at the slits taken out of $S_\infty$ and that $ h_+(w)\in \Omega_+$ cannot take values in $2n\pi i$. Thus
  $\varphi_\pm$ is also analytic in the slit disk $D(u,r)\setminus[-2,+\infty)$.

Substituting $w=2(z-e^z)$, we see that
\begin{eqnarray*}
\lim_{w\to -2,\; w\in S_\infty}(w+2)\varphi_\pm(w)&=&
\lim_{w\to -2, \; w\in S_\infty}e^{k_0 h_\pm(w)}\frac{w+2}{(e^{h_\pm(w)}-1)^2}
\\
&=&2\lim_{z\to 0,\;  z\in \Omega_\pm}e^{k_0z} \frac{1+z-e^z}{(e^z-1)^2}\ =\ -1.
\end{eqnarray*}

\end{proof}

%W changed "heading of proof" to start with "Conclusion of"
\begin{proof}[Conclusion of proof of Proposition \ref{P 1D} for $\alpha=2$]
We now prove \eqref{1D Pressure} and identify the limit. Using functions $\varphi_\pm$ constructed above, define
 \begin{equation}\label{tilde G+}
   \widetilde{G}(z,\la):=\begin{cases}
     (e^\la-1)^2\varphi_+(z (e^\la -1)^2+ 2\la - 2 e^\la) & \mbox{ if $\la>0$ }\\
(e^\la-1)^2\varphi_-(z (e^\la -1)^2+ 2\la - 2 e^\la) & \mbox{ if $\la<0$ }.
   \end{cases}
 \end{equation}
  From Lemma \ref{LC1} we see that function $\widetilde{G}(z,\la)$
  satisfies \eqref{generating_function_pde}  for all $(z,\la) \in\mathcal{D}_+\cup\mathcal{D}_- $.
By uniqueness of the PDE solution in each of the two regions, % $\mathcal{D}_-$ and $\mathcal{D}_+$,
 $G(z,\la)=\widetilde{G}(z,\la)$ for all $(z,\la)\in \mathcal{D}_+\cup\mathcal{D}_-$.
In particular, for each fixed $\la\ne 0$,
$\limsup_{n\to\infty}(m_n(\la))^{1/n}$ is the reciprocal of the
radius of convergence of the series expansion of
$\widetilde{G}(z,\lambda)$ at $z=0$. The latter  is $(1-e^\la)^{-2}$ times
the radius of convergence for $\varphi_{sign(\la)}(w)$ at
$u=2\la-2 e^\la\in(-\infty,
-2)$ which by Lemma \ref{L2} is   $2(e^\la-\la-1)$. Furthermore, the lemma implies that there is $\eta=\eta(\la)>0$ such
that %W corrected English
 $\varphi_{sign(\la)}(w)$ is analytic on the slit disk
$\{w: |w-u|<2(e^\la-\la-1)(1+\eta)\}\setminus [-2,\infty)$. Since after appropriate translation and re-scaling  this slit disk is larger than the  indented disk
 $\Delta(\pi/4,\eta)$ introduced in \cite[(2.5)]{Flajolet-Odlyzko-90}, and the second part of Lemma \ref{L2} gives  $\varphi_\pm(w)\approx -1/(w+2)$ as $w\to -2$, we can apply \cite[Corollary 2]{Flajolet-Odlyzko-90}
 to get \eqref{1D Pressure}.

 Finally, as $m_n(0)=1$  the
 convergence at $\la=0$
 holds trivially.
\end{proof}
%
%%%%%
%% End of section 2.1

%%%%%%%%%%%
%% Path level Large Deviations
%%%%%%%%%%%

\section{Proof of  Theorem \ref{T2}}\label{PT2}

We follow the method and notation of Dupuis-Ellis
\cite{Dupuis-Ellis-97}.  Although some arguments are analogous to
those found in \cite[Chapter 6]{Dupuis-Ellis-97} which considers random walk models with time-homogeneous continuous statistics, and \cite{Zhang-Dupuis-08} where a different model with
a time singularity at $t=0$ is examined, for completeness we give all relevant details as several differ, especially in
the lower bound proof.

Let $X^n_{j} = Z_{j-k_0+1}/n$ for
$k_0\leq j\leq n$, and set $X^n_j=j/n$ for $0\leq j\leq k_0$. Then, noting
(\ref{Z-def}), given $X^n_j$, we have $X^n_{j+1} = X^n_j + v^n_j/n$
where $v^n_j$ has Bernoulli distribution
$\rho_{\sigma_n(j/n),X^n_j}$.  Here,
\begin{eqnarray*}
\sigma_n(t)&=&\begin{cases}
s_{\lfloor nt-k_0+1\rfloor}/n & \ {\rm for \ } t\geq \frac{k_0}{n}\\
0 & \ {\rm for \ } t< \frac{k_0}{n}
\end{cases}\\
\rho_{\sigma,x}(A)& = & \frac{x}{\sigma}\delta_0(A) + \left(1-\frac{x}{\sigma}\right)\delta_1(A)\ \ \ \  \ {\rm for\ }A\subset \RR,\
0\leq x\leq \sigma,
\end{eqnarray*}
%S put the qualifiers in display trying to make it look better
 and
$\rho_{0,0} := \delta_1$.

Define $X^n_{\cdot}$ as the polygonal interpolated path connecting
points $(j/n, X^n_j)$ for $0\leq j\leq n$.
 Also, for probability measures $\mu
\ll\nu$ such that $d\mu=f d\nu$, denote $R(\mu\|\nu)=\int f \log
fd\nu$,
%S put d\nu in entropy integral above
 the relative entropy; set $R(\mu\|\nu)=\infty$
 when $\mu$ is not absolutely continuous with respect to $\nu$.

Let $h: C([0,1],\RR)\rightarrow \RR$ be a bounded continuous
function. To prove Theorem \ref{T2}, we need only establish Laplace
principle upper and lower bounds \cite[page 74]{Dupuis-Ellis-97}.
The upper bounds are to show
$$\limsup_{n\rightarrow\infty} \frac{1}{n}\log
E\big\{\exp[-nh(X^n_\cdot)]\big\} \ \leq \ - \inf_{\varphi\in
  \mathcal C_1}\big\{I(\varphi) + h(\varphi)\big\}$$
for a rate function $I$ and a
closed subset of continuous functions $\mathcal C_1$.
The lower bounds are to prove the reverse inequality
$$\liminf_{n\rightarrow\infty} \frac{1}{n}\log
E\big\{\exp[-nh(X^n_\cdot)]\big\} \ \geq \ - \inf_{\varphi\in
  \mathcal C_1}\big\{I(\varphi) + h(\varphi)\big\}.$$

Define, for $k_0+1\leq j\leq n$, noting $X^n_j = j/n$ for $j\leq k_0$
is deterministic, that
 $$ W^n(j,x_j,\ldots,x_{k_0+1}) \ =
\ -\frac{1}{n}\log
E\Big[e^{-nh(X^n_\cdot)}\Big|X^n_j=x_j,\ldots,X^n_{k_0+1}=x_{k_0 +1}
\Big],$$ and
$$W^n\ :=\ W^n(k_0,\emptyset)\ =\  -\frac{1}{n}\log
E\Big[e^{-nh(X^n_\cdot)}\Big].$$ Then, by the Markov property, for $k_0+1\leq j\leq
n-1$,
\begin{eqnarray*}
e^{-nW^n(j,x_j,\ldots,x_{k_0+1})} &=&  E\Big[e^{-nW^n(j+1,X^n_{j+1}, x_j,\ldots,x_{k_0+1})}\Big|X^n_j=x_j,\ldots,X^n_{k_0+1}=x_{k_0+1}\Big]\\
&=&\int e^{-nW^n(j+1,x_j +
v/n,x_j,\ldots,x_{k_0+1})}d\rho_{\sigma_n(j/n),x_j}(dv).
\end{eqnarray*}

By a property of relative entropy \cite[Prop. 1.4.2
(a)]{Dupuis-Ellis-97}, for $k_0+1\leq j\leq n-1$,
\begin{eqnarray*}
&&W^n(j,x_j,\ldots,x_{k_0+1})\\
&&= \  -\frac{1}{n}\log\int
e^{-nW^n(j+1,x_j+y/n,x_j,\ldots,x_{k_0+1})}\rho_{\sigma_n(j/n),x_j}(dv)\\
&&= \ \inf_\gamma\Big\{\frac{1}{n}R(\gamma\|
\rho_{\sigma_n(j/n),x_j}) + \int
W^n(j+1,x_j+y/n,x_j,\ldots,x_1)\gamma(dy)\Big\}.
\end{eqnarray*}
Also, the boundary condition $W^n(n,x_n,\ldots,x_{k_0+1})  =
h(x_\cdot)$ holds with respect to the linearly interpolated path
$x_\cdot=x^n_\cdot$ connecting $\{(l/n,l/n)\}_{0\leq l\leq k_0}$,

The basic observation in the Dupuis-Ellis method is that
$W^n(j,x_j,\ldots,x_{k_0+1})$ satisfies a control problem
(\cite[section 3.2]{Dupuis-Ellis-97})
whose solution for $k_0\leq j\leq n-1$ %W changed n to n-1
 is
$$V^n(j,x_j,\ldots,x_{k_0+1}) \ = \ \inf_{\{v^n_i\}}\bar{E}_{j,x_j,\ldots,x_{k_0+1}}
\Big\{\frac{1}{n}\sum_{i=j}^{n-1}R(v^n_i(\cdot)\|
\rho_{\sigma_n(i/n),\bar{X}^n_i}) + h(\bar{X}^n_\cdot)\Big\}.$$
Here, $v^n_i(dy) = v^n_i(dy ; x_{k_0},\ldots,x_i)$ is a Bernoulli
distribution given $x_{k_0},\ldots,x_i$ for $k_0\leq i\leq n-1$ and
in the display $\nu^n_i(\cdot)= \nu^n_i(dy|\bar{X}^n_{k_0},\ldots,\bar{X}^n_i)$;
$\{\bar{X}^n_i; 0\leq i\leq n
\}$ is the adapted path satisfying
$\bar{X}^n_l= l/n$ for $0\leq l\leq k_0$, and $\bar{X}^n_{l+1} =
\bar{X}^n_l + \bar{Y}^n_l/n$ for $k_0\leq l\leq n-1$ where
$\bar{Y}^n_l$, conditional on $(\bar{X}^n_l,\ldots
\bar{X}^n_{k_0})$ has distribution $v^n_l(\cdot)$;
$\bar{X}^n_\cdot$ is the interpolated path with respect to
$\{\bar{X}^n_l\}$; and $\bar{E}_{j,x_j,\ldots,x_{k_0+1}}$ denotes
conditional expectation with respect to the $\bar{X}^n_\cdot$
process given the values $\{\bar{X}^n_l = x_l: k_0+1\leq l\leq j\}$
at step $k_0+1\leq j\leq n$. The boundary conditions are $V^n(n,x_n,\ldots,x_{k_0+1}) = h(x_{\cdot})$ and
\begin{equation}\label{V}
V^n(k_0,\emptyset) = V^n =
\inf_{\{v^n_j\}}\bar{E}\Big\{\frac{1}{n}\sum_{j=k_0}^{n-1}R(v^n_j(\cdot)\|
\rho_{\sigma_n(j/n),\bar{X}^n_j}) + h(\bar{X}^n_\cdot)\Big\}.
\end{equation}

In particular, by \cite[Corollary 5.2.1]{Dupuis-Ellis-97},
\begin{equation}\label{DE-main}
W^n \ =\ - \frac{1}{n}\log E\Big[e^{-nh(X^n_\cdot)}\Big] \ =\
V^n.\end{equation}  The goal will be to take Laplace limits using
this representation.  To simply later expressions, we will take
$v^n_j = \delta_1$ for $0\leq j\leq
k_0-1$ when $k_0\geq 1$.

\subsection{Upper bound}
To establish the upper bound, we first put the controls $\{v^n_j\}$
into continuous time paths: Let $v^n(dy|t)  = v^n_j(dy)$ for
$t\in (j/n,(j+1)/n], \; 0\leq j \leq n-1$, and $v^n(dy|0)= v^n_0$.  Define
$$v^n(A\times B) \ = \ \int_B v^n(A|t)dt$$ for $A\subset \RR$,
$B\subset [0,1]$. Define also the piecewise constant path
$\tilde{X}^n(t)  =  \bar{X}^n_j$
for $t\in (j/n,(j+1)/n], \; 0\leq j\leq n-1$, and $\tilde{X}^n(0)=0$.
 Then,
\begin{eqnarray}\label{V=inf}
V^n & = & \inf_{\{v^n_j\}}\bar{E}\Big\{\int_{0}^1 R(v^n(\cdot|t)\|
\rho_{\sigma_n(\lfloor nt\rfloor/n),\tilde{X}^n(t)})dt +
h(\bar{X}^n_\cdot)\Big\}.
\end{eqnarray}
From this control representation, as $|V^n|=|W^n|\leq \|h\|_\infty$ and $\rho_{\sigma,x}$ is supported on $K=\{0,1\}$, for
each $n$, there is $\{v^n_j\}$ supported on $K$ and corresponding $v^n(dy\times dt)
=v^n(dy|t)\times dt$ such that
$$W^n +\epsilon \ =\  V^n + \epsilon \ \geq \ \bar{E}\Big[\int_{0}^1
R(v^n(\cdot|t)\|\rho_{\sigma_n(\lfloor nt\rfloor/n),\tilde{X}^n(t)})dt +
h(\bar{X}^n_\cdot)\Big].$$

As the sets $K$ and
\begin{eqnarray*}
\Gamma  \ =\   \Big\{ \varphi\in C([0,1];\RR):\; \varphi(0)=0,\;
\varphi\uparrow,
\ {\rm Lipschitz, \ with\  bound\ } 1\Big\}
\end{eqnarray*}
are compact on $\RR$ and $C([0,1],\RR)$ respectively, and
$\{v^n_j\}$ are probability measures on $K$
and $\{\bar{X}^n_\cdot\}\subset \Gamma$, by
Prokhorov's theorem, the distributions of
  $(v^n(dy\times dt),\bar{X}^n_\cdot)$ have a subsequence which converges
  weakly to
  a limit distribution governing $(v,\bar{X}_\cdot)$ where for each realization,
$v$ is a probability measure on $K\times [0,1]$ and $\bar{X}_\cdot\in
\Gamma$. More precisely, let $(\bar{\Omega},\bar{F},\bar{P})$ be a
probability space where $v:\bar{\Omega}\rightarrow {\rm probability \
  measures \ on \ }K\times [0,1]$, and
$\bar{X}:\bar{\Omega} \rightarrow \Gamma$. Then, \cite[Lemma
3.3.1]{Dupuis-Ellis-97} gives that $v$ is the subsequential weak limit of
$v^n$, and $\bar{P}$-a.s. for $\omega\in \bar{\Omega}$,
$$v(A\times B|\omega) \ =\  \int_B v(A|t,\omega)dt$$
for some kernel $v(dy|t,\omega)$ on $K$ given $t$ and $\omega$.

Now, by the same proof given for \cite[Theorem 5.3.5]{Dupuis-Ellis-97} (only
\cite[equation (5.12)]{Dupuis-Ellis-97} in the theorem statement differs; in our context $\mu$ there is
replaced by $\rho_{\sigma_n(j/n),\tilde{X}^n_j}$),
 as $K$ is
compact, we have $(v^n,\bar{X}^n, \tilde{X}^n)$ has a subsequential
limit in distribution to $(v,\bar{X},\bar{X})$ [the last coordinate
with respect to $D([0,1],\RR)$].  Also, $\bar{P}$-a.s., for all
$t\in [0,1]$,
$$\bar{X}(t) \ =\  \int_{\RR\times [0,1]}yv(dy\times ds)
\ =\  \int_0^t\int_K y v(dy|s)ds.$$ In particular, $\bar{P}$-a.s.,
$\dot{\bar{X}}(t) = \int_Kyv(dy|t)$.

By Skorohod representation theorem, we can assume
$(v^n,\bar{X}^n,\tilde{X}^n) \rightarrow (v,\bar{X},\bar{X})$
converges a.s.  In particular, $\bar{X}^n$ converges uniformly to
$\bar{X}$ a.s., and it is clear that $\tilde{X}^n$ converges
uniformly to $\bar{X}$ a.s. as $\bar{X}$ is continuous (\cite[p.
154]{Dupuis-Ellis-97}).

Let $\mu_\delta(\cdot) = (1-\delta)^{-1}1_{[\delta,1]}(\cdot)$ and $\nu^n(dy\times \mu_\delta dt) = \nu^n(dy|t)\times \mu_\delta dt$.
Then, by \cite[Lemma 1.4.3 (f)]{Dupuis-Ellis-97} we have, for
$0<\delta<1$, that
\begin{eqnarray*}
&&\int_\delta^1 R\big(v^n(\cdot|t)\big\|\rho_{\sigma_n(\lfloor
nt\rfloor/n),\tilde{X}^n(t)}\big)dt\\
 &&\ \ \ \ \ \ \ \ \  = \ (1-\delta)R\big(v^n(dy\times
\mu_\delta dt)
\Big\| \rho_{\sigma_n(\lfloor nt\rfloor/n),\tilde{X}^n(t)}
\times \mu_\delta dt\big).\end{eqnarray*} 
Write,
%SS removed an extra line of space between write and the eqnarray
\begin{eqnarray*}
\liminf_{n\to\infty}V^n &\geq& \liminf_{n\rightarrow\infty} \bar{E}\Big[\int_{\delta}^1
R\big(v^n(\cdot|t)\big\|\rho_{\sigma_n(\lfloor nt\rfloor/n),\tilde{X}^n(t)}\big)dt
+h(\bar{X}^n_\cdot)\Big]\\
&\geq&\liminf_{n\rightarrow\infty}%\\
%&&\
\bar{E}\Big[(1-\delta)R\Big(v^n(dy\times \mu_\delta dt)
\Big\|
\rho_{\sigma_n(\lfloor nt\rfloor/{n}),\tilde{X}^n(t)}
\times \mu_\delta dt\Big) +h(\bar{X}_\cdot)\Big]\\
&\geq& \bar{E}\Big[\int_\delta^1 R(v(\cdot|t)\|\rho_{\alpha
t,\bar{X}(t)})dt + h(\bar{X}_\cdot)\Big]\end{eqnarray*} where in the last
line we use Fatou's lemma, noting lower semi-continuity of $R$, $\nu^n(dy\times \mu_\delta dt) \rightarrow \nu(dy|t)\times \mu_\delta dt$ a.s., and
$\rho_{\sigma_n(\lfloor nt\rfloor/n),\tilde{X}^n(t)}$ converges
in distribution to $\rho_{\alpha t,\bar{X}(t)}$ for $t\in [\delta,1]$ a.s. since $\sigma_n(\lfloor nt\rfloor/n) \rightarrow \alpha t$,
$\tilde{X}^n(t)\rightarrow \bar{X}(t)$ uniformly on $[0,1]$ a.s., and $\rho_{\sigma,x}$ is continuous on
$\{(\sigma,x): \delta\leq \sigma \leq 1, 0\leq x\leq \sigma\}$ (cf. Section 6.2 \cite{Dupuis-Ellis-97}).

By \cite[Lemma 3.3.3(b)]{Dupuis-Ellis-97},
$$R(v(\cdot|t)\| \rho_{\alpha t,\bar{X}(t)}) \ \geq \ L\left(\alpha t,\bar{X}(t), \int y
v(dy|t)\right)$$ where
\begin{eqnarray}L(t,x,y ) &=& \sup_\theta
\Big\{ \theta y   -  \log \int e^{\theta z}
\rho_{t,x}(dz)\Big\}.\label{Ldef}\end{eqnarray}
We note, for
$t>0$,
$L( t,x,y)$ diverges when $x=0$, $0\leq y<1$, and $x=t$, $0<y\leq 1$ but is finite
when $0< x< t$ and $0\leq y\leq 1$, and in this case evaluates to
\begin{eqnarray}
L(t,x,y)&=&y   \log \left(\frac{ t y  }{ t-x}\right)+(1-y  )
   \log \left(\frac{ t (1-y )}{x}\right),
\label{LL}
\end{eqnarray}
understood with the usual convention $0\log(0) = 0$.

Since $\int y v(dy|t) = \dot{\bar{X}}(t)$, we have
$$\liminf_{n\to\infty}V^n \ \geq\  \bar{E}\Big[\int_\delta^1 L(\alpha t,\bar{X}(t),\dot{\bar{X}}(t))dt +
h(\bar{X}_\cdot)\Big].$$
As $L\geq 0$,
and $\bar X(\cdot)\in\Gamma$,
letting $\delta\downarrow 0$, we obtain,
$$\liminf_{n\to\infty}V^n\ \geq\  \inf_{\varphi\in\Gamma}\int _0^1 L(\alpha t,\varphi(t),\dot \varphi(t))dt+h(\varphi).$$

Taking into account \eqref{DE-main},
the upper bound holds with
$\mathcal C_1=\Gamma$ and
$$I(\varphi) \ =\  \int_0^1 L(\alpha t,\varphi(t),\dot{\varphi}(t))\, dt.$$

%%%%%%%%%%%%%%%%%%%
%% LOWER bound
%%%%%%%%%%%%%%%%%%%

\subsection{Lower bound}
In the following, for typographical convenience, we write
 $E(X;\mathbb{A})$
for $\int_\mathbb{A} X dP$.  Let now $\psi^*\in \Gamma$ be such that
$I(\psi^*)<\infty$, and fix a bounded, continuous (in the sup
norm) function $h: C([0,1];\RR)\to \RR$. In view of \eqref{DE-main},
we need only show, for each $\eps>0$, that
\begin{equation}\label{LD LB goal}
\limsup_{n\to\infty}V^n\ \leq \ I(\psi^*)+h(\psi^*)+8\eps.
\end{equation}
%Note that $0<\psi^*(t)<t$ for almost all $t\in(0,1)$.
\medskip
{\it Step 1.} Our first goal is to replace $\psi^*$ by its appropriate
regularization. We use the trick of considering a convex combination
of paths,
% (cf. \cite{Dupuis-Nuzman-Whiting-04})
$$\psi_\theta(t)\ =\ (1-\theta)\psi^*(t)+\theta t$$
for $0\leq \theta\leq 1$.
Since $\|\psi_\theta-\psi^*\|_\infty\leq 2\theta$, it is clear,
for small enough $\theta>0$, that $|h(\psi_\theta)-h(\psi^*)|<\eps$.
Further, since $I$ is finite on the line $y=t$ with slope $1$, by convexity
arguments \cite[Lemma 1.4.3 (b)]{Dupuis-Ellis-97},  for small enough
$\theta>0$,
$$I(\psi_\theta)\ \leq\  I(\psi^*)+\eps. $$
We therefore fix $\theta>0$ such that  $I(\psi_\theta)<I(\psi^*)+\eps$ and $h(\psi_\theta)<h(\psi^*)+\eps$.

Next,  following \cite[p. 82]{Dupuis-Ellis-97}, we write
\begin{equation}\label{phi_kappa}
\varphi_\kappa(t) \ = \ \int_0^t \gamma_\kappa(s)ds,
\end{equation}
where $$\gamma_\kappa(t) \ =\  \kappa
\int_{i/\kappa}^{i+1/\kappa}\dot{\psi_\theta}(s)ds$$
for $t\in (i/\kappa, (i+1)/\kappa]$, $0\leq i\leq \kappa-1$, and $\gamma_\kappa(0)=\gamma_\kappa(1/\kappa)$.
For large enough $\kappa$, we have
\begin{eqnarray}
 \frac{\varphi_\kappa(t)}{t}&\geq& \theta\ >\ 0 \label{away 0}\\
 h(\varphi_\kappa) &\leq&h(\psi^*)+2\eps \label{h kappa}\\
 I(\varphi_\kappa) &\leq&I(\psi^*)+2\eps. \label{I kappa}
\end{eqnarray}
Inequality \eqref{away 0} is a property of $\psi_\theta$, and is preserved by \eqref{phi_kappa}.
Since
$$\lim_{\kappa\rightarrow \infty}\sup_{t\in [0,1]}|\psi_\theta(t) -
\varphi_\kappa(t)| \ =\  0,$$
inequality \eqref{h kappa} follows from continuity of $h$ by our choice of $\theta$.
Then, as $$\dot{\varphi}_\kappa (t) \ =\  \kappa\int_{\lfloor
t\kappa\rfloor/\kappa}^{\lceil t\kappa\rceil/\kappa}
\dot{\psi_\theta}(s)ds \ \rightarrow\  \dot{\psi_\theta}(t)$$ as $\kappa\uparrow
\infty$, and
$(t,x,y)\mapsto L(\alpha t,x,y)$
(cf. (\ref{Ldef}),(\ref{LL})) as $\alpha>1$
 is bounded, uniformly continuous on the compact set
$$\Big\{(t,x,y):\; 0\leq t\leq 1,\; \theta t\leq x\leq t,\; \theta\leq y\leq 1\Big\},$$
by the dominated convergence theorem, $\lim_\kappa I(\varphi_\kappa)= I(\psi_\theta)$.
 Inequality \eqref{I kappa} follows due to our choice of $\theta$.
%\end{proof}

We now fix $\kappa$ such that the above bounds hold.

Since \eqref{I kappa} implies $I(\varphi_\kappa)<\infty$,
we now choose a $0<\delta<1/3$ such that
\begin{equation}\label{alpha}
\int_0^\delta L(\alpha t,\varphi_\kappa(t),\dot\varphi_\kappa(t))dt\
<\ \eps.\end{equation}

We will also need an estimate on $s_n$.  From assumptions \eqref{s_n},
there is an $0<\eta<1/2$ and $k_1\geq k_0+1$
such
that for $n\geq k_1$,
\begin{equation}\label{eta}
\eta\ \leq\ \frac{n}{s_{n-k_0+1}}\ \leq \ 1-\eta.
\end{equation}
With respect to $\eta$ and $\theta$, we impose additionally
that $\delta$ satisfies
\begin{equation}\label{alpha2}
-\delta \log(\delta\eta\theta)\ <\ \eps/10.\end{equation}
\medskip
{\it Step 2.}
We now build a sequence of controls based on $\varphi_\kappa$.
Recall that we already set $v^n_j(dy)=\delta_1$ for $0\leq j\leq k_0-1$ when $k_0\geq 1$.
Define
$$v^n_j(dy;x_{k_0},\dots,x_j) \ =\
\begin{cases}
\rho_{\sigma_n(j/n),{x}_j}& \ {\rm for \ }k_0\leq j \leq k_1\\
\rho_{1,1-\dot\varphi_\kappa(j/n)} &  \ {\rm for \ }j\geq k_1+1.
\end{cases}
$$
%S put for in display above
Note, for $j\geq k_1+1$, $\nu^n_j$ does not depend on auxiliary inputs $x_{k_0},\ldots, x_j$, and is in fact the Bernoulli distribution with success probability $\dot{\varphi}_\kappa (j/n)$.

Define also $\bar{X}^n_l =l/n$ for $0\leq l\leq k_0$, and $\bar{X}^n_{j+1} = \bar{X}^n_j
+\bar{Y}^n_j/n$ for $j\geq k_0$ where
$$\bar{P}(\bar{Y}^n_j\in dy|\bar{X}^n_0,\ldots, \bar{X}^n_j)\  =\
v_j(dy;\bar{X}^n_0,\ldots, \bar{X}^n_j).
$$
  Thus, for $j\geq 1$, $\bar{X}^n_j =
(1/n)\sum_{\ell=0}^{j-1}Y^n_\ell$ where $\{Y^n_j\}_{j\geq k_1+1}$ are
independent Bernoulli random variables with corresponding means
$\{\dot{\varphi}_\kappa (j/n)\}_{j\geq k_1+1}$.

\medskip \noindent
{\it Step 3.} We now collect some useful estimates.\vskip .1cm
\noindent A.  Since $\bar Y^n_0\equiv 1$ (when $k_0=0$, recall $\rho_{\sigma,0}=\delta_1$), and the increments are at most one,
 we have $1/n\leq \bar{X}^n_j\leq
j/n$ for $1\leq j\leq n$. \vskip .1cm

 \noindent B. We have
\begin{equation}\label{as}
\lim_{n\uparrow \infty} \sup_{0\leq j/n\leq 1} \left|\bar{X}^n_j
-\frac{1}{n}\sum_{l=0}^j \dot{\varphi}_\kappa(l/n)\right| \ =\  0\ \ \ \ \
{\rm a.s.}
\end{equation} Indeed, for large enough $n$, as $0\leq \dot\varphi_\kappa\leq 1$,
 $(1/n)\sum_{j=0}^{k_1} |\dot \varphi_\kappa(j/n)|+k_1/n<n^{-1/8}$.
 Then, by Doob's maximal inequality,
\begin{eqnarray}
&&\bar{P}\Big[\sup_{0 \leq j/n\leq 1} \Big|\bar{X}^n_j
-\frac{1}{n}\sum_{l=0}^j\dot{\varphi}_\kappa(l/n)\Big|>
 \frac{\theta}{2n^{1/8}}
 \Big]\nonumber\\
&& \ \ \ \ \leq\   Cn^{1/2}\bar{E}\Big|\bar{X}^n_n -
\bar E(X_n^n)\Big|^4\nonumber\\
&&\ \ \ \ \leq \ \frac{1}{ n^{7/2}}\Big[C\sum_{l=0}^n
\dot\varphi_\kappa(l/n) (1- \dot\varphi_\kappa(l/n))(1-3
\dot\varphi_\kappa(l/n)+3 \dot\varphi_\kappa^2(l/n))\nonumber\\
&&\ \ \ \ \ \ \ \ \ \ \ \ \ \ \ \ \ +\ C\left(\sum_{l=0}^n \dot\varphi_\kappa(l/n)
(1- \dot\varphi_\kappa(l/n))\right)^2\Big] \ \leq\  {C}
n^{-3/2}.\label{Rosenthal}\end{eqnarray} where $C$ is a
constant changing line to line.

\vskip .1cm \noindent C.  For $0\leq j\leq n$, from \eqref{away 0},
it follows that
\begin{equation}\label{LB phi}
\frac{1}{n}\sum_{l=0}^j \dot{\varphi}_\kappa (l/n)\ \geq \ \theta
\frac{j}{n}.
\end{equation}

\vskip .1cm \noindent D.
Let $j\geq k_1+1$.
Noting $1/n\leq \bar{X}^n_j\leq j/n$ (cf. part A)
 and bounds \eqref{eta}, we have
$$\eta \leq  1-\frac{j}{s_{j-k_0+1}} \leq 1-\frac{\bar{X}^n_j}{\sigma_n(j/n)}\leq 1, \ \ \ \frac{1}{s_{j-k_0+1}}\leq \frac{\bar{X}^n_j}{\sigma_n(j/n)}\leq 1-\eta.$$
Hence, $L(\sigma_n(j/n),
\bar{X}^n_j, \dot{\varphi}_\kappa (j/n))$ can be well evaluated (cf.
(\ref{Ldef}),  (\ref{LL})), and we may rewrite the relative entropy as
\begin{eqnarray*}
R(v^n_j\| \rho_{\sigma_n(j/n),\bar{X}^n_j})
&=& \dot{\varphi}_\kappa (j/n)\log[\dot{\varphi}_\kappa (j/n)/(1-\bar{X}^n_j/\sigma_n(j/n))]\\
&&\ \ \ +
(1-\dot{\varphi}_\kappa (j/n))\log\big((1-\dot{\varphi}_\kappa (j/n))/(\bar{X}^n_j/\sigma_n(j/n))\big)
\\
&=&
L(\sigma_n(j/n),
\bar{X}^n_j, \dot{\varphi}_\kappa (j/n)).
\end{eqnarray*}
Further, as $0\leq \dot
\varphi_\kappa \leq 1$, we have
\begin{eqnarray}
&&L(\sigma_n(j/n),
\bar{X}^n_j, \dot{\varphi}_\kappa (j/n))\nonumber\\
&&\ \ \ \ \  =\ \big[\dot
\varphi_\kappa(j/n)\log\dot\varphi_\kappa(j/n)+
(1-\dot \varphi_\kappa(j/n))\log(1-\dot\varphi_\kappa(j/n))\big]\nonumber\\
&&\ \ \ \ \ \ \ \ \ \ \ \ \   -\dot
\varphi_\kappa(j/n)\log\big(1-{\bar{X}_j^n}/{\sigma_n(j/n)}\big)-
 (1-\dot \varphi_\kappa(j/n))\log\big({\bar X_j^n}/{\sigma_n(j/n)}\big)\nonumber\\
&&\ \ \ \ \ \leq \  0- \log\eta+\log (s_{j-k_0+1})\  \leq\ \log(j/\eta^2),
\label{lb_eta_j_estimate}\end{eqnarray}
the last inequality using (\ref{eta}) again.

\medskip
{\it Step 4.} We now argue (\ref{LD LB goal}) via representation (\ref{V}). Let
$$\mathbb A \ = \ \Big\{\sup_{0\leq j\leq n-1}|\bar{X}^n_j
-\frac{1}{n}\sum_{l=0}^j
\dot{\varphi}_\kappa (l/n)|>\frac{\theta}{2n^{1/8}}\Big\}.$$
Since
$R(v^n_j\| \rho_{\sigma_n(j/n),\bar{X}^n_j}) =0$ for $k_0\leq j\leq
k_1$, the sum in (\ref{V}) equals
\begin{eqnarray}
\bar{E}\Big[\frac{1}{n} \sum_{j=k_0}^{n-1} R(v^n_j\|
\rho_{\sigma_n(j/n), \bar{X}^n_j})\Big] & = &
\bar{E}\Big[\frac{1}{n}\sum_{j=k_1+1}^{n-1}
 L(\sigma_n(j/n),\bar{X}^n_j, \dot{\varphi}_\kappa (j/n))\Big]\nonumber\\
&=& \bar{E}\Big[ \frac{1}{n} \sum^{n-1}_{j=k_1+1} L(\sigma_n(j/n),
\bar{X}_j^n, \dot{\varphi}_\kappa (j/n));
\mathbb A \Big]\nonumber\\
&&\  \  + \ \bar{E}\Big[ \frac{1}{n} \sum_{j=k_1+1}^{n-1}
L(\sigma_n(j/n), \bar{X}_j^n, \dot{\varphi}_\kappa (j/n));\mathbb A^c
\Big]\nonumber \\
& = & A_1 + A_2. \label{Split}\end{eqnarray}
%for an $0<\epsilon<1$.
\medskip
{\it Step 5.} We now treat the first term $A_1$ in (\ref{Split}).
Combining (\ref{lb_eta_j_estimate}) with \eqref{Rosenthal}, we
obtain, for large $n$,
\begin{eqnarray}
&&\bar{E}\Big[ \frac{1}{n} \sum^{n-1}_{j=k_1+1} L(\sigma_n(j/n),
\bar{X}_j^n, \dot{\varphi}_\kappa (j/n));{\sup_{0\leq j\leq
n-1}|\bar{X}^n_j
-\frac{1}{n}\sum_{l=0}^j \dot{\varphi}_\kappa (l/n)|>\frac{\theta}{2n^{1/8}}} \Big]\nonumber\\
&&\  \leq \ \log (n/\eta^2) \bar{P}\Big[\sup_{0 \leq j/n\leq 1}
\Big|\bar{X}^n_j
-\frac{1}{n}\sum_{l=0}^j\dot{\varphi}_\kappa (l/n)\Big|>\frac{\theta}{2n^{1/8}}\Big]\nonumber\\
&&\  \leq\  C \log (n/\eta^2) n^{-3/2}\ < \ \eps .\label{away}\end{eqnarray}
\medskip

{\it Step 6.} For the other term $A_2$ in \eqref{Split}, we split it
into two sums depending on when index $j\leq \delta n$ or $j\geq
\delta n$ (recall $\delta$ from \eqref{alpha}):
\begin{eqnarray}
&&\bar{E}\Big[ \frac{1}{n} \sum_{j=k_1+1}^{\lfloor\delta n\rfloor}
L(\sigma_n(j/n), \bar{X}_j^n, \dot{\varphi}_\kappa (j/n)); \mathbb A^c\Big]\nonumber \\
%\sup_{0\leq j\leq n-1}|\bar{X}^n_j -\frac{1}{n}\sum_{l=0}^j
%\dot{\varphi}_\kappa (l/n)|\leq \frac{\theta}{2n^{1/8}}\Big] \nonumber\\
&&\ \ \ \ \ \ \ \ \ \ \ \ \  \ +\ \bar{E}\Big[ \frac{1}{n} \sum_{j=\lceil\delta n\rceil}^{n-1} L(\sigma_n(j/n),
\bar{X}_j^n, \dot{\varphi}_\kappa (j/n)); \mathbb A^c\Big]\nonumber \\
%\sup_{0\leq j\leq
%n-1}|\bar{X}^n_j -\frac{1}{n}\sum_{l=0}^j
%\dot{\varphi}_\kappa (l/n)|\leq \frac{\theta}{2n^{1/8}}\Big] \nonumber\\
&&=\ B_1 + B_2
\label{B1+B2}.\end{eqnarray}
\medskip
{\it Step 7.} To estimate $B_1$, we divide it further into two terms
corresponding to sums on indices $j\leq n^{7/8}$ and $n^{7/8}\leq j\leq \delta
n$:
\begin{eqnarray*}
&&\bar{E}\Big[ \frac{1}{n} \sum^{\lfloor n^{7/8}\rfloor }_{j=k_1+1} L(\sigma_n(j/n),
\bar{X}_j^n, \dot{\varphi}_\kappa (j/n)); \mathbb A^c \Big]\\
&&\ \ \ \ \ \ \ \ \ \ \ \ \  \ +\ \bar{E}\Big[ \frac{1}{n}
\sum^{\lfloor \delta n\rfloor }_{j= \lceil n^{7/8}\rceil} L(\sigma_n(j/n), \bar{X}_j^n,
\dot{\varphi}_\kappa (j/n));\mathbb A^c \Big]\\
&&\ \ = \ D_1 + D_2.
%\label{small_sum}
\end{eqnarray*}

The first
term $D_1$,
% (\ref{small_sum})
using (\ref{lb_eta_j_estimate}), is bounded for all large $n$ by
\begin{equation}
\label{step7} n^{-1/8} \log (n^{7/8}/\eta^2)\ <\ \eps.\end{equation}

The second term $D_2$,
%in (\ref{small_sum}),
using again
(\ref{lb_eta_j_estimate}),
is bounded in absolute value by
\begin{eqnarray*}
&& -\delta \log \eta+\frac1n\sum_{j=\lceil n^{7/8}\rceil }^{\lfloor \delta n\rfloor} |\log (s_{j-k_0 +1}/n)|\\
&&\ \ - \bar E\Big[\frac{1}{n}\sum_{j=\lceil n^{7/8}\rceil}^{\lfloor \delta n\rfloor }
\log \bar X_j^n
; \sup_{0\leq j\leq n-1}|\bar X_j^n-
\frac1n\sum_{l=0}^j\dot\varphi_\kappa(l/n)|\leq \frac{\theta}{2n^{1/8}}\Big].
\end{eqnarray*}
Now, note that \eqref{eta} implies for $j\geq k_1 +1$ that
$$\frac{j}{(1-\eta)n}\ \leq\  \frac{s_{j-k_0+1}}{n}\ \leq\  \frac{j}{\eta n}.$$
Then, as $0<\delta<1/3$ and $0<\eta<1/2$, we have for large $n$ that
\begin{eqnarray*}
\frac1n\sum_{j=\lceil n^{7/8}\rceil }^{\lfloor \delta n\rfloor} |\log (s_{j-k_0 +1}/n)|&\leq&
2\int_0^\delta \max\Big\{\Big|\log\Big(\frac{x}{1-\eta}\Big)\Big|,\Big|\log\Big(\frac{x}{\eta}\Big)\Big|\Big\}\,dx\\
&\leq& -6\delta\log \delta -2\delta\log\eta.\end{eqnarray*}
Also, noting
\eqref{LB phi}, we have for large $n$ that
\begin{eqnarray*}
&& - \bar E\Big[\frac{1}{n}\sum_{j=\lceil n^{7/8}\rceil}^{\lfloor \delta n\rfloor }
\log \bar X_j^n
; \sup_{0\leq j\leq n-1}|\bar X_j^n-
\frac1n\sum_{l=0}^j\dot\varphi_\kappa(l/n)|\leq \frac{\theta}{2n^{1/8}}\Big]\\
&&\ \ \leq \ -\frac{1}{n}\sum_{j=
\lceil n^{7/8}\rceil}^{\lfloor\delta n\rfloor} \log \Big[\frac{\theta j}{n} -\frac{\theta}{2n^{1/8}}\Big]
\\
&&\ \ \leq \ -\frac{1}{n}\sum_{j=\lceil
n^{7/8}\rceil}^{\lfloor \delta n\rfloor} \log \Big[\frac{\theta j}{2n}\Big] \ \leq \ -2\int_0^\delta \log\frac{\theta x}{2}\, dx\ \leq \ -4\delta\log \delta -2\delta\log \theta.
\end{eqnarray*}
By combining these estimates, we have $D_2$ is bounded by a function of $\delta,\eta,\theta$ which, given \eqref{alpha2}, can be made small:
\begin{equation}
\label{3.19}
D_2 \ \leq \ -10\delta\log \delta - 3\delta\log\eta -2\delta\log \theta \ < \ \eps.
\end{equation}

\medskip\noindent
{\it Step 8.} We now estimate the second term $B_2$ in
\eqref{B1+B2}. Note,
 for $ n > \delta^{-8}$,
  by \eqref{LB phi}
  the event
\begin{equation*}
\Big\{\sup_{0\leq j\leq n-1}|\bar{X}^n_j -\frac{1}{n}\sum_{l=0}^j
\dot{\varphi}_\kappa (l/n)|\leq\frac{\theta}{2n^{1/8}}\Big\}\\
 \ \subset \ \Big\{\inf_{j\geq \delta n} \bar{X}^n_j \geq
\frac{\delta\theta}{2}\Big\}.\end{equation*} Hence, for large $n$,
\begin{eqnarray*}
B_2 &\leq&
% &&\bar{E}\Big[\frac{1}{n}\sum_{j=\delta n}^{n-1}
% L(\sigma_n(j/n),\bar{X}^n_j, \dot{\varphi}_\kappa (j/n));
% \sup_{0\leq j\leq n-1}|\bar{X}^n_j -\frac{1}{n}\sum_{l=0}^j
%\dot{\varphi}_\kappa (l/n)|\leq\frac{1}{n^{1/8}}\Big]\\
\bar{E}\Big[\frac{1}{n}\sum_{j=\lceil\delta n\rceil}^{n-1}
 L(\sigma_n(j/n),\bar{X}^n_j, \dot{\varphi}_\kappa (j/n));\mathbb A^c \cap \Big\{\inf_{j\geq \delta n} \bar{X}^n_j \geq
\frac{\delta\theta}{2} \Big\}\Big].
\end{eqnarray*}

Note, from assumption (\ref{s_n}) and $\bar{X}^n_j\leq j/n$,
that
%S put that in sentence
$\delta/(1-\eta)\leq \sigma_n(j/n)\leq 1/\eta$
and $\bar{X}^n_j\leq \sigma_n(j/n) (1-\eta)$ when $\delta\leq j/n \leq 1$ for all large $n$.
Also,
 $L( t,x,y)$ is continuous, and therefore also bounded and
uniformly continuous on the compact set (cf. definition of $L$ (\ref{Ldef}), (\ref{LL})),
\begin{equation}
\label{compact_set}
\Big\{(t,x,y):\; \frac{\delta}{1-\eta}\leq t\leq
1/\eta,\; \frac{\delta\theta}{
%S put back enlarged set--I was worried that we are using uniform continuity with respect to
%S X^n_j and \sum \dot\varphi(l/n), both have to be in the compact set which is true for large enough n
%S certainly with the larger compact set
%W OK, thanks for clarification
4}\leq x \leq (1-\eta/2)t,\;
0\leq y\leq 1\Big\}.
\end{equation}
Then,
\begin{eqnarray}
&&\limsup_{n\rightarrow\infty} \bar{E}\Big[\frac{1}{n}\sum_{j=\lceil\delta
n\rceil}^{n-1}
 L(\sigma_n(j/n),\bar{X}^n_j, \dot{\varphi}_\kappa (j/n));\mathbb A^c\cap \Big\{\inf_{j\geq \delta n} \bar{X}^n_j \geq
\frac{\delta\theta}{2} \Big\}\Big]\nonumber\\
&& \ \ \ \ \ \ \ \leq\  \limsup_{n\to\infty}\frac{1}{n}\sum_{j=\lceil\delta n\rceil}^{n-1}
 L\Big(\sigma_n(j/n),\frac{1}{n}\sum_{l=0}^j\dot\varphi_\kappa(l/n), \dot{\varphi}_\kappa (j/n)\Big).
\label{lb_last_piece}\end{eqnarray} Further,
\begin{equation}
\label{last_supremum_phi}\lim_{n\rightarrow \infty} \sup_{1\leq
j\leq n} \Big|\frac{1}{n}\sum_{l=0}^j \dot{\varphi}_\kappa (l/n)
-\int_0^{j/n} \dot{\varphi}_\kappa (s)ds\Big| \ =\ 0,\end{equation}
as $\dot\varphi_k$ is piecewise constant and bounded.
Then, given the bounds on $\sigma_n(j/n)$ above,
%S reworked this last paragraph to make more clear
%\lfloor nt\rfloor/n) \rightarrow \alpha t$ for $\alpha>1$,
 $\theta j/n\leq (1/n)\sum_{l=0}^{j} \dot\varphi_k(l/n)\leq j/n$ (cf. (\ref{LB phi})),
 and uniform continuity of $L$ on the compact set (\ref{compact_set}),
%S put a number on the compact set
we may analogously bound (\ref{lb_last_piece}) by
\begin{eqnarray}
\lim_{n\to\infty}\frac{1}{n}\sum_{j=\lceil\delta n\rceil}^{n-1}
 L(\sigma_n(j/n),\varphi_\kappa(j/n), \dot{\varphi}_\kappa (j/n))
\ =\
 \int_\delta^1L(\alpha t,\varphi_\kappa(t),\dot\varphi_\kappa(t))dt.
\label{step8}\end{eqnarray}
\medskip
{\it Step 9.}  Finally, with respect to the second term in
(\ref{V}), by (\ref{as}) and (\ref{last_supremum_phi}), in the sup
topology, $\lim_{n\rightarrow \infty}h(\bar{X}^n_\cdot)  =
h(\varphi_\kappa(\cdot))$.

We now combine all bounds to conclude the proof of \eqref{LD LB
goal}. By \eqref{V}, and bounds \eqref{away}, \eqref{step7}, \eqref{3.19}, \eqref{step8}, we have
\begin{eqnarray*}
\limsup_{n\to\infty} V^n&\leq&\limsup_{n\to\infty}
\bar{E}\Big[\frac{1}{n}\sum_{j=k_1+1}^{n-1}
 L(\sigma_n(j/n),\bar{X}^n_j, \dot{\varphi}_\kappa (j/n)) + h(\bar{X}^n_\cdot)\Big]\\
&\leq& \ 3\eps+\int_\delta^1L(\alpha
t,\varphi_\kappa(t),\dot\varphi_\kappa(t))\, dt + h(\varphi_k).
\end{eqnarray*}
Then, by \eqref{alpha}, \eqref{h kappa}, and  \eqref{I kappa}, we
obtain \eqref{LD LB goal}.

%%%%%%%%%%%%%%
% END of lower bound
%%%%%%%%%%%%%%%

\section{Proofs of Theorems \ref{T3} and \ref{T4}}
\label{ODE_section}

We first address the proof of Theorem \ref{T3}, and later the proof
of Theorem \ref{T4} in Subsection \ref{proof_of_T4}.
 Since $Z_n/n=\frac{n+k_0-1}{n}X_{n+k_0-1}(1)$, by contraction principle
Theorem  \ref{T2} implies Theorem \ref{T3} with the rate function $I$ given in (\ref{contracted_rate}).
As $x\mapsto \la x$ is a bounded continuous function on $[0,1]$,  by
Varadhan's Integral Lemma (see \cite[Theorem
4.3.1]{Dembo-Zeitouni-98}, or \cite[Theorem
1.3.4]{Dupuis-Ellis-97}), this implies that the limit \eqref{1D
Pressure} exists, and equals
\begin{eqnarray}
  \La(\la)&=&\sup_\varphi\{ \la \varphi(1)-I(\varphi)\}\nonumber\\
    &=&\sup_{\varphi: \varphi(0)=0}\Big\{ \int_0^1\big[\la \dot\varphi (t)
  -\dot\varphi (t)\log\frac{\alpha t \dot\varphi (t)}{\alpha
  t-\varphi(t)}\nonumber\\
  &&\ \ \ \ \ \ \ \ \ \ \ \ \ \ \ \ \ \ \ \ \ \ \ \ \  -
  (1-\dot\varphi (t))\log \frac{\alpha t (1-\dot\varphi
  (t))}{\varphi(t)}\big]\,dt\Big\}.
\label{PL}\end{eqnarray}
 Direct derivation of formula \eqref{Pressure_2}
 or even formula \eqref{R-def}
 for $\alpha=2$ from \eqref{PL} seems quite challenging (cf. \cite{Zhang-Dupuis-08}).
The  Euler equations are:
\begin{eqnarray*}
  \frac{\ddot{\varphi}}{\dot\varphi (1-\dot\varphi )}&=&\frac{\alpha}{\alpha t -\varphi}-\frac{1}{\varphi} \label{E1}\\
  \varphi(0)&=&0\\
  \frac{\dot\varphi (1)}{1-\dot\varphi (1)}&=& \frac{\alpha-\varphi(1)}{\varphi(1)}e^\la .
  %\label{E3}.
\end{eqnarray*}
Numerical evidence suggests that the solutions of the Euler equations indeed give the correct answer.

\subsection{Proof of Theorem \ref{T3}}\label{sec0}

In this section we show that one can use Theorem \ref{T2} to set up the differential equation (\ref{ODE}) which implies
formula \eqref{R-def}. Recall notation
for the moment generating function.
As we already noted, Theorem \ref{T2} implies that
\begin{equation}
\label{LLL} \frac{1}{n}\log m_n(\lambda) \ \rightarrow \
\La(\lambda)
\end{equation}
 with $\La(\la)$ given by \eqref{PL}.

Formula \eqref{R-def} follows from the following additional fact.
%%%%%
\begin{proposition}\label{P1} $\La(\la)$, as defined by \eqref{LLL}, is differentiable and satisfies
equation \eqref{ODE}.
\end{proposition}
%The initial condition is of course $\La(0)=0$.
It is straightforward to verify that \eqref{R-def} satisfies
\eqref{ODE}. The following argument shows the uniqueness:
 Suppose $\La_1,\La_2$ are two solutions with initial condition $\La(0)=0$. If $\La_1(t)=\La_2(t)$ for
 some $t>0$,
then they coincide for all $t>0$. Therefore, we must have $\La_1(t)>\La_2(t)$ for all $t>0$.
By the mean value theorem,
there is $t_0>0$ such that $\La_1'(t_0)-\La_2'(t_0)>0$.
But the equation gives $\La_1'(t)-\La_2'(t)=\alpha \frac{e^{\La_1(t)}-e^{\La_2(t)}}{1-e^t}<0$ for all $t>0$.
Similarly,  if $\La_1(t)=\La_2(t)$ for some $t<0$,
then they coincide for all $t<0$. Therefore, we must have $\La_1(t)>\La_2(t)$ for all $t<0$.
By the mean value theorem,
there is $t_0<0$ such that $\La_1'(t_0)-\La_2'(t_0)=\frac{\La_1(t)-\La_2(t)}{t}<0$.
But the equation gives $\La_1'(t)-\La_2'(t)=\alpha \frac{e^{\La_1(t)}-e^{\La_2(t)}}{1-e^t}>0$ for all $t<0$.

\subsection{Proof of Proposition \ref{P1}}
As
%S took out already
mentioned in the introduction, the differential equation is easy to derive heuristically from
\eqref{recursion_divided_through}; the main technical difficulty is in justifying the convergence of various expressions.

To this end, the key
%S changed main to key
ingredient is the control of the complex zeroes of $m_n(z)$, based on (the proof of)  \cite[Theorem 1]{bona-2007}.  Note that the assumptions in the following Proposition \ref{CL4.1} include $s_n$ satisfying (\ref{s_n}), and $s_n = \alpha n$ for $\alpha = 1/2,1$ with $k_0\leq s_1$.
\begin{proposition}\label{CL4.1}
Suppose $m_n(z)$ is defined by \eqref{recursion} with initial condition $m_1(\la)=e^{k_0\la}$
such that $\{s_n\}$ satisfies
$s_1\geq k_0$,
$s_n\geq\max\{ k_0,1\}$ for $n\geq 2$.
We also assume that there is $n_0\geq 1$ such that $s_{n_0}>k_0$ when $k_0=1,2,\dots$ and
that there is $n_1\geq 2$ such that $s_{n_1}>1$ when $k_0=0$.
Then $m_n(z)\ne 0$ in the strip $|\Im(z)|<\pi$.
\end{proposition}
%%%%%%%%%%%%%%%%%%%%%%%%%%%
%% Revised proof of Lemma 4.2
%%  \label{CL4.1}
%%%%%%%%%%%%%%%%%%%%%%%%%%%
\begin{proof}
We note that  $m_n(\la)=p_n(e^\la)$ for a polynomial $p_n(u)$  with
non-negative coefficients. Then, to prove the result, we repeat the
proof of \cite[Theorem 1]{bona-2007} to deduce that all complex
zeros of $p_n(u)$ are real,  and non-negative for $n\geq 1$. [Some
minor details differ from \cite{bona-2007}.]

Since $m_n'(\la)=up_n'(u)$,  \eqref{recursion}  gives
\begin{eqnarray}
p_{n+1}(u)&=&\frac{u(1-u)}{s_n}p_n'(u)+up_n(u)\nonumber \\
&=&\frac{u}{s_n}(1-u)^{s_n+1}
\frac{d}{du}\left[(1-u)^{-s_n}p_n(u)\right]. \label{Bona}
\end{eqnarray}

Note that  $p_1(u)=u^{k_0}$ and that $s_1\geq k_0$. We first give the proof for the case when $s_1> k_0\geq 1$.
To reach our conclusion we prove by induction the following statement.
\begin{quote} For $n\geq 2$,
$p_n(u)$ is a polynomial of degree $n-1+k_0$ with a root of multiplicity $k_0$ at $u=0$ and $n-1\geq 1$
   strictly negative simple roots.
\end{quote}
As  $s_1>k_0$, polynomial $p_{2}(u)=\frac{1}{s_1}u^{k_0} ((s_{1}-k_0)u+ k_0)$
satisfies the inductive statement.
Suppose for some $n\geq 2$, polynomial  $p_n(u)$ is of degree $n-1+k_0$, has $n-1$ simple negative roots
 $u_1<u_2<\dots<u_{n-1}$ and
  a root of multiplicity $k_0$ at $u=0$. Then, from the first equality in \eqref{Bona},
  $p_{n+1}(u)$ has also a root of multiplicity $k_0$ at $u=0$.
  Clearly,   $\{u_1,u_2,\dots,u_{n-1},0\}$ are $n$ distinct  roots  of the expression
  under the derivative on the right hand side of \eqref{Bona}. Since $\{u_j\}$ are simple roots,
  the function must cross the real line, so by  Rolle's theorem,
$p_{n+1}(u)$ has $n-1$ distinct roots interlaced between the roots   of $p_n(u)$. This shows that $p_{n+1}(u)$ has  $n+k_0-1$ real roots
 in the interval $(u_1,0]$.

To end the proof, we want to show that the last $(n+k_0)$th root of
$p_{n+1}(u)$ is located to the left of $u_1$,
so that all negative roots of $p_{n+1}(u)$ must be simple. To see this, we follow
again \cite{bona-2007}: The first equation in \eqref{Bona} shows
that $p_{n+1}(u_1)$ and $p'_n(u_1)$ have opposite signs, and
$p'_n(u_1)\ne0$ as $u_1$ is a simple root. So $p_{n+1}(u_1)\ne 0$.
Since polynomials $p_{n+1}(u)$ and $p_n(u)$ have positive leading
coefficients and their degrees  differ by $1$, their signs are
opposite as $u\to-\infty$. Since $u_1$ is the smallest root,
$p_n(u)$ has constant sign for $u<u_1$ which matches the sign of
$p_{n+1}(u_1)$. Thus $p_{n+1}(u)$ must eventually cross the real
line to the left of $u_1$. This shows that $p_{n+1}(u)$ has $n$
simple strictly negative roots, and a   root of multiplicity $k_0$
at $u=0$, ending the induction proof.

Next, suppose $s_1=s_2=s_{n_0-1}=k_0\geq 1$,  but $s_{n_0}>k_0$ for some $n_0\geq 2$.
Then $p_1(u)=p_2(u)=\dots=p_{n_0}(u)=u^{k_0}$   and the inductive proof goes through with minor modifications, starting with  $p_{n_0+1}=\frac{1}{s_{n_0}}u^{k_0} ((s_{n_0}-k_0)u+ k_0)$ that replaces $p_2(u)$ in the previous
argument.

Finally, if $k_0=0$, then $p_1(u)=1$, $p_2(u)=u$. Choose first $n_1\geq 2$ such that  $s_{n_1}>1$
but $s_n=1$ for $2\leq n<n_1$.
Since  in this case the value of $s_1$ is irrelevant,
we get
$p_{2}(u)=\dots=p_{n_1}(u)=u$.
 The induction proof proceeds with minor modifications, starting with $p_{n_1+1}(u)=\frac{1}{s_{n_1}}u ((s_{n_1}-1)u+ 1)$.
\end{proof}
%%%%%%%%%%%%%%
%% end of revised proof
%%%%%%%%%%%%%%

\begin{lemma}\label{CL4.3} $\La$ is differentiable and $\frac{m_n'(\la)}{nm_n(\la)}\to\La'(\la)$.
\end{lemma}
%%%%%%
% proof based on David's notes
%%%%%%
\begin{proof}
Recall that if $f(\zeta)$ is holomorphic in $|\zeta|\leq R$, then for $0<r<R$
\begin{equation}\label{Caratheodory}
\max_{|\zeta|\leq r} |f(\zeta)|\leq \frac{R+r}{R-r}|f(0)|+\frac{2r}{R-r}\max_{|\zeta|\leq R}\Re f(\zeta)
\end{equation}
(cf. \cite[Theorem 6.31(ii)]{Burckel-79}).

  By Proposition \ref{CL4.1}, for fixed  $n\geq 1$, the function $m_n(z)$ is holomorphic and nonzero in the strip
$  |\Im z|<\pi$.  Since $m_n(z)$ is a polynomial in $e^z$ with nonnegative coefficients, $m_n(t)>0$ for all $t\in\RR$,
$m_n(t)$ is increasing on $\RR$, and $|m_n(z)|\leq m_n(|z|)$.

The function $m_n(z)$ has a holomorphic logarithm $L_n(z)$   on the strip $  |\Im z|<\pi$. Because  $m_n(t)>0$ for
$t\in\RR$, we may assume that $L_n(t)=\log(m_n(t))$ for all $t\in\RR$. For each $t\in\RR$, we can apply \eqref{Caratheodory} to $f(\zeta)=n^{-1}L_n(t+\zeta)$, $|\zeta|<\pi$, with $R=2\pi/3$ and $r=\pi/3$. This gives
\begin{eqnarray*}
\max_{|z-t|\leq \pi/3}n^{-1}|L_n(z) | &\leq&  3|n^{-1}L_n(t)|+2n^{-1}\max_{|z-t|<\pi/3}|\Re L_n(\zeta)|\\
&=& 3 |n^{-1}\log(m_n(t))|+ {2}{n^{-1}}\max_{|z-t|<\pi/3}\log |m_n(z)|
 \\
& \leq & 3 |n^{-1}\log(m_n(t))|+ 2\log( m_n(|t|+2\pi/3)).\end{eqnarray*}
Since $e^{-n|t|}\leq m_n(t)\leq e^{n|t|}$,
 $\{n^{-1}L_n(z): n\geq 1\}$
 %S put \geq 1 instead 1,2,\dots\}$
 is a normal family (i.e., a uniformly bounded family of holomorphic functions) in the disk $|z-t|\leq \pi/3$.

We now note that  $\{n^{-1}L_n(t)\}$ converges for all real $t$; this holds true by \eqref{LLL} when $\alpha>1$, or by
Proposition \ref{P 1D} when $s_n=n$
or $s_n=n/2$. A version of Vitalli's theorem, see
% \cite[Theorem 15.3.4)]{Hille-62} or
\cite[p. 9]{Duren-83},  implies that $n^{-1}L_n(z)\to\La(z)$ in the strip $|\Im  z| \leq \pi/3$,
 the convergence is
uniform in each disc $ |z-t | \leq \pi/3 $,
 the limit  $\La(z)$ is an analytic function of the argument $z$ in that strip, and all derivatives of $n^{-1}L_n$ converge to the corresponding derivatives of $\La$.
In particular, the sequence
$n^{-1}m_n'(\lambda)/m_n(\lambda)=n^{-1}L_n'(\la)$ converges to $\La'(\la)$
 for all real $\la$.
\end{proof}
%\end{proof}
%%%%%%
% end of revision
%%%%%%

\begin{proof}[Conclusion of Proof of Proposition \ref{P1}]
By Lemma \ref{CL4.3} the right hand side
of  \eqref{recursion_divided_through}
converges, so the left hand side must converge too:
$\frac{m_{n+1}(\lambda)}{m_n(\lambda)}\to \exp L(\la)$ for some $L$ uniformly in a neighborhood of $\la$.
Since the limit of ratios implies the same limit for $n$-th roots, we get
$\frac{1}{n}\log m_n(\lambda) \ \rightarrow \ L(\lambda)$, which identifies
$L(\la)=\La(\la)$ as the pressure. From
Lemma \ref{CL4.3},
the derivative $m_n'(\lambda)/(nm_n(\lambda))\to \La'(\la)$, so passing to the limit in
\eqref{recursion_divided_through}
one obtains
the differential equation for the
pressure.
\end{proof}
\subsection{Proof of Theorem \ref{T4}}
\label{proof_of_T4}
In the following, $s_n$ satisfies assumptions \eqref{s_n},  or $s_n =n$ or $s_n=n/2$.
The LLN follows from the strict convexity of $I$ in Theorem \ref{T3},
and $\La'(0)=\alpha/(1+\alpha)$.

For the CLT, we recall  \cite[Proposition 2]{Bryc93c} in our
context:  When $\sup_{n} m_n(\epsilon)^{1/n} <\infty$ for
some $\epsilon>0$, $0\not\in {\rm closure}(\cup_{n\geq 1} Ze_n)$
where $Ze_n$ is the zero set of $m_n(z)=E[\exp\{zZ_n\}]$, and
$Z_n/n$ satisfies an LDP, then $(Z_n - E[Z_n])/\sqrt{n}$ converges
in distribution to $N(0,\sigma^2)$ where $\sigma^2 = \Lambda''(0)$.

To verify assumptions, note by Theorem \ref{T3} that the LDP for $Z_n/n$ holds and
$\lim_{n\rightarrow \infty}n^{-1}\log m_n(\lambda) =
\Lambda(\lambda)$ for all $\lambda\in \mathbb R$,
and by
Proposition \ref{CL4.1}
that $m_n(z)$ has no zeroes in the strip
$|\Im(z)|<\pi$ for $n\geq 1$. Finally,
$\La''(0)={\alpha^2}/[{(1+\alpha)^2(2+\alpha)}]$ to finish the
proof.
\qed

\section{Concluding remarks}
 \label{Concludingremarks}

We now comment on some possible extensions.

\medskip
1. {\it Cases $\alpha = 1/2,1$.}  Although we prove a LDP for $Z_n/n$ when
$s_n$ is linear with slopes $\alpha =1/2, 1$
(Remark \ref{rmk1}, Proposition \ref{P 1D}),
the proof
of Theorem \ref{T3} for a path LDP with respect to $Z_{\lfloor nt\rfloor}/n$,
especially the lower bound argument, does not cover these cases.
The difficulty is in controlling
boundary behavior as estimate (\ref{eta}) is not available.  It
would be interesting to look further into these issues.
\medskip

2. {\it Higher order statistics.}  With respect to random graph
models, one might ask about LDP's for the vector ${\bf Z}^k_n/n =
\langle Z_1(n), \ldots, Z_k(n)\rangle/n$ where the $j$th component $Z_j(n)$
counts the number of vertices with degree $j\leq k$ for $k\geq 2$.
In principle, our method to analyze the leaves can be used to study
${\bf Z}^k_n/n$.  Indeed, the Dupuis-Ellis type arguments given here for a
path LDP for the leaves $Z_1(\lfloor nt\rfloor)/n$ (Theorem \ref{T3}) would seem
to extend to the vector-valued paths ${\bf Z}^k_{\lfloor nt\rfloor}/n$.

 However, to calculate
the pressure $${\bf \Lambda}_k(\lambda_1,\ldots,\lambda_k) \ = \
\lim_{n\rightarrow\infty} \frac{1}{n} \log
E\Big[\exp\Big\{\sum_{i=1}^k \lambda_i Z_i(n)\Big\}\Big],$$ as in
Theorem \ref{T2} for the leaves, the differential equation which now
arises for ${\bf \Lambda}_k$, in place of the ODE for $\Lambda$
(\ref{ODE}), is a quasilinear PDE with $k\geq 2$ independent variables.
%For instance, with respect to preferential attachment graphs, we
%obtain
These PDE's, although in principle implicitly solved by the
method of characteristics, unfortunately do not seem to admit
explicit solutions, at least to the extent found here with respect $Z_1(n)/n$,
a reason why we have focused on detailed investigations on the leaves.
It would be of interest to study better these higher order questions.

\subsection*{Acknowledgements}
We thank Professor S.R.S. Varadhan for illuminating discussions.

%W.~B. research was   supported in part by Taft Research Seminar 2008/09. D.~M. research was
% supported in part by Taft Research Seminar 2006/07 and 2008/09. S.~S. research was supported in part by
% NSF
%grant DMS-0504193.

%%%%  BibTeX
%  Bib file needs updating!!! (not to be used for now)
%%% applied probab trust
%\bibliographystyle{apt}
%\bibliography{leaves}
%\end{document}

\def\cprime{$'$}

\end{document}